\documentclass[11pt]{amsart}  
\usepackage{amssymb,latexsym}
\usepackage{amssymb,latexsym}
\usepackage{amsfonts,mathrsfs}
\usepackage[margin=1.4in]{geometry}
\usepackage{fancyhdr}
\usepackage{hyperref}
\usepackage{xcolor}

\newtheorem{theorem}{Theorem}[section]
\newtheorem{lemma}[theorem]{Lemma}
\newtheorem{proposition}[theorem]{Proposition}
\newtheorem{corollary}[theorem]{Corollary}

\theoremstyle{remark}
\newtheorem{remark}{Remark}[section]

\def\Xint#1{\mathchoice
{\XXint\displaystyle\textstyle{#1}}%
{\XXint\textstyle\scriptstyle{#1}}%
{\XXint\scriptstyle\scriptscriptstyle{#1}}%
{\XXint\scriptscriptstyle\scriptscriptstyle{#1}}%
\!\int}
\def\XXint#1#2#3{{\setbox0=\hbox{$#1{#2#3}{\int}$ }
\vcenter{\hbox{$#2#3$ }}\kern-.6\wd0}}
\def\dashint{\Xint-}
\newcommand{\Int}{\mathrm{Int}}
\newcommand{\Vol}{\mathrm{Vol}}

\makeatletter

\newcommand{\Rmnum}[1]{\expandafter\@slowromancap\romannumeral #1@}
\makeatother
\begin{document}
\allowdisplaybreaks
\title{Flow by powers of the Gauss  curvature in space forms}
\thanks{\noindent \textbf{MR(2010)Subject Classification}   35K55, 35B65, 53A05, 58G11}
\author{Min Chen and Jiuzhou Huang}
\address{1203 Burnside Hall, 805 Sherbrooke Street West Montreal, Quebec H3A 0B9}
\email{min.chen5@mail.mcgill.ca}
\address{1030 Burnside Hall, 805 Sherbrooke Street West Montreal, Quebec H3A 0B9}
\email{jiuzhou.huang@mail.mcgill.ca}
\thanks{The first author is supported by the National Nature Science Foudation of China  No. 11721101 and National Key Research and Development Project No. SQ2020YFA070080. }
\pagestyle{fancy}
\fancyhf{}
\renewcommand{\headrulewidth}{0pt}
\fancyhead[CE]{}
\fancyhead[CO]{\leftmark}
\fancyhead[LE,RO]{\thepage}

\begin{abstract}
In this paper,  we prove that convex hypersurfaces under the flow by powers $\alpha>0$ of the Gauss curvature in space forms $\mathbb{N}^{n+1}(\kappa)$ of constant sectional curvature $\kappa$ $(\kappa=\pm 1)$ contract to a point in finite time $T^*$. Moreover, convex hypersurfaces under the flow by power $\alpha>\frac{1}{n+2}$ of the Gauss curvature converge (after rescaling) to a limit which is the geodesic sphere in $\mathbb{N}^{n+1}(\kappa)$. This extends the known results in Euclidean space to space forms.
\end{abstract} 

\maketitle
\numberwithin{equation}{section}
\section{Introduction}
Parabolic flows for hypersurfaces play important roles in geometric analysis. One important example is the flow by Gauss curvature. Much attention has been paid to flow of convex hypersurfaces $X(\cdot,\tau): M\rightarrow \mathbb{R}^{n+1}$ by power of Gauss curvature:
\begin{equation}\label{unnorm flow on Euclidean space}
	X_\tau(x,\tau)=-K^\alpha(x,\tau)\nu(x,\tau) , \quad \alpha>0,
\end{equation}
 where $\nu(x,\tau)$ is the unit exterior normal at $X(x,\tau)$ of $M_{\tau}=X(M,\tau)$ and $K(x,\tau)$ is the Gauss curvature of $M_{\tau}$.

Gauss curvature flow was introduced by Firey \cite{F} to model the shape of tumbling stones. It was proved in \cite{Tso} for $\alpha=1$, and in \cite{C nroot} for any $\alpha>0$  that the flow shrinks to a point in finite time $T^*>0$ for any smooth strictly convex initial hypersurfaces $M$.  A  Harnack type inequality for Gauss curvature flow of compact convex hypersurfaces for all $\alpha>0$, and an entropy estimate for $\alpha=1$ were proved in \cite{C entropy}. Hamilton \cite{H} used these results to get the sharp upper bound of Gauss curvature and the diameter. The main interest is to understand the asymptotic behavior of the flows (\ref{unnorm flow on Euclidean space}) as the time $\tau$ approaches to the singular time $T^*$. 

When $n=1, \alpha =1$, (\ref{unnorm flow on Euclidean space} is the curve shorting flow, convergence to circle was proved  by Gage-Hamilton \cite{Gage-Hamilton} for initial convex curve, and Grayson \cite{Grayson} for general initial curve. Convergence to circles  was proved for $n=1$ and $\alpha>1$ in \cite{curves}, for $n=1$ and $\frac{1}{3}<\alpha<1$ in \cite{isotropic} with convex initial curve.  For general $n>1$, Chow \cite{C nroot} analyzed the case $\alpha=\frac{1}{n}$ and proved that solutions of the normalized flow converge to the unit sphere as $t\rightarrow \infty$. The convergence to sphere when $n=2,\alpha=1$ was established by Andrews in \cite{A invention}, see also \cite{BC} for the case $n=2$ and $\frac{1}{2}<\alpha <1$. The exponent $\alpha=\frac{1}{n+2}$ is critical as it's the affine curvature flow. In this case, the convergence to ellipsoids was established by Andrews in \cite{affine} (see also  \cite{ST} for $n=1$). Convergence to solitons was established for $\alpha\in(\frac{1}{n+2},\frac{1}{n})$ in \cite{A pinch} for a family of anisotropic Gauss curvature flows (more general situation). For the normalized flow of (\ref{unnorm flow on Euclidean space}) with strictly convex initial hypersurfaces in $\mathbb R^{n+1}, \forall n\ge 1$, the convergence to solitons (self-similar solutions) was established for the case $\alpha=1$ by Guan-Ni \cite{Guan-Ni}, and by Andrews-Guan-Ni \cite{AGN} for $\forall \alpha>\frac{1}{n+2}$. In \cite{AGN}, the uniqueness of soliton (round sphere) was proved when it is centrally symmetric. The final resolution of the uniqueness of solitons of normalized flow of (\ref{unnorm flow on Euclidean space}) were obtained  by Choi-Daskalopoulos in \cite{Choi} ($\frac{1}{n}< \alpha <1+\frac{1}{n}$) and by Brendle-Choi-Daskalopoulos \cite{BCD}  for all $\alpha> \frac{1}{n+2}$. 

 Parabolic flows for hypersurfaces in general Remannian manifolds were considered by many authors. Generalization of flows by mean curvature in Euclidean space to general Riemannian manifold \cite{Huisken-R} was a fundamental contribution by Huisken. More recently, a new type of mean curvature flow in space forms was introduced by Guan and Li \cite{Guan-Li}. Gerhardt \cite{Gerhardt sphere} demonstrated a correspondence between contracting and expanding flows of hypersurfaces in the sphere. Andrews, Han, Li and Wei \cite{Andrews-Han-Li-Wei} generalized Andrew's noncollapsing estimates for curvature flows in Euclidean space to fully nonlinear curvature flows in space forms. 

It is natural to consider flows by powers of Gauss curvature in more general ambient spaces. Very little is known except for the case $\alpha=1$, $n=2$ or  $\alpha =1$, $n\geq 3$ and initial hypersurfaces are axially symmetric \cite{McCoy}. 
                                                                                                                                                                                                                                                                                                                                                                                                                                                                                                                                                                                                                                                                                                                                                                                                                                                                                                                                                                                                                                                                                                                                                                                                                                                                                                                                                                                                                                                                                                                                                                                                                                                                                                                                                                                                                                                                                                                                                                                                                                                                                                                                                                                                                                                                                                                                                                                                                                                                                                                                                                                                                                                                                                                                                                                                                                                                                                                                                                                                                                                                                                                                                                                                                                                                                                                                                                                                                                                           
In this paper, we establish complete analogous results of flow by power of Gauss curvature in space forms: 
\begin{equation}\label{unnorm flow on spaceform}
	\left\{\begin{split}
	\tilde{X}_\tau(x,\tau)&=-\tilde{K}^\alpha(x,\tau)\nu(x,\tau),\\
	\tilde{X}(0)&=\tilde{X}_0,
	\end{split}
	\right.
\end{equation}
where $\nu(x,\tau)$ is the unit exterior normal at $\tilde{X}(x,\tau)$ and $\tilde{K}(x,\tau)$ is the Gauss curvature of $\tilde{M}_{\tau}$, $\mathbb{N}^{n+1}(\kappa)$ is the $(n+1)$ dimensional simply connected space form of constant sectional curvature $\kappa=\pm 1$ (the tildes distinguish these from the normalized counterparts introduced later). Below is our main theorem.
\begin{theorem}\label{main theo} 
If $\tilde{X}_0$ represents a strictly convex smooth hypersurface in $\mathbb{N}^{n+1}(\kappa)$, then for any $\alpha>0$, the initial value problem (\ref{unnorm flow on spaceform}) has a unique solution on a maximum finite time interval $[0,T^*)$ such that the $\tilde M_{\tau}$ converges to a point as $\tau \rightarrow T^*$. Moreover,  for $\alpha>
\frac{1}{n+2}$, $\tilde{M}_{\tau}$ converges to a geodesic sphere in $\mathbb{N}^{n+1}(\kappa)$ in the $C^{\infty}$-topology after re-scaling.
\end{theorem}
The theorem generalizes the known results in Euclidean space to space forms. The first statement is a generalization of \cite{Tso,C nroot}. The second statement extends results in \cite{ Guan-Ni, AGN, BCD}. 
 
Our approach to flow (\ref{unnorm flow on spaceform})  is to deduce it to a flow in the Euclidean space by proper projections.  Choosing the projection $\pi_p$ (see details in section 2), the Gauss curvature of the image satisfies (\ref{KKhat rela}). It is suffice to consider the following type of flow (the image of projection) in Euclidean space:
\begin{equation}\label{X-unnor}
\hat{\tilde{X}}_{\tau}(x,\tau)=-(1+\kappa|\hat{\tilde{X}}|^2)^{\frac{n+2}{2}\alpha+\frac{1}{2}}(1+\kappa\langle \hat{\tilde{X}},\hat\nu \rangle^2)^{-\frac{n+2}{2}\alpha+\frac{1}{2}}	\hat{\tilde{K}}^\alpha,
\end{equation}
where $\kappa=1$ when $\tilde{X}(x,\tau)$ is the flow of convex hypersurfaces in $\mathbb{S}^{n+1}$ and where $\kappa=-1$ when $\tilde{X}(x,\tau)$ is the flow of convex hypersurfaces in $\mathbb{H}^{n+1}$ (the hat distinguish these from the counterparts before the projection). It is well known that any strictly convex hypersurfaces $\hat{\tilde{M}}$ in $\mathbb{R}^{n+1}$ can be recovered completely from the support function $u$ by $\hat{\tilde{M}}=\{\nabla \hat{\tilde{u}}+\hat{\tilde{u}}x, x\in \mathbb{S}^{n}\}$, see e.g. \cite{Schneider}. Then the support function satisfies  equation
\begin{equation}\label{rn+1 unor support1}
\begin{aligned}
	\hat{\tilde{u}}_\tau(x,\tau)=&-(1+\kappa(\hat{\tilde{u}}^2+|\nabla \hat{\tilde{u}}|^2))^{\frac{n+2}{2}\alpha+\frac{1}{2}}(1+\kappa\hat{\tilde{u}}^2)^{-\frac{n+2}{2}\alpha+\frac{1}{2}}{\det}^{-\alpha}(\nabla^2 \hat{\tilde{u}}+\hat{\tilde{u}}I).
	\end{aligned}
\end{equation}
For flow (\ref{rn+1 unor support1}), we obtain  the estimates of the lower bound of principal curvature and the upper bound Gauss curvature and pinching estimate of the inner and outer radii. 

The key in our proof is an almost monotonicity formula for associated entropies considered in  \cite{ Guan-Ni, AGN}. In this respect, the normalized flow  (\ref{norm u evo general}) of (\ref{rn+1 unor support1}) will be used in section 4. A crucial observation is the  decay estimate (\ref{decayest}) in Section 4. It allows us to obtain a monotone quantity $\mathcal{E}_{\alpha}(\hat{\Omega}_t)+C(n,\alpha,\tilde{X}_0)e^{-\frac{2(n+1)}{2n+1}t}$ along the normalized flow (\ref{norm u evo general}) by modifying the monotone quantity used in \cite{Guan-Ni, AGN}. From this, we can use the methods in \cite{Guan-Ni, AGN} to obtain a uniformly lower and upper bound of support function. This in turn implies a uniform $C^2$-estimate, and to conclude that the normalized flow for any smooth initial convex body converges smoothly as $t\rightarrow \infty$ to a uniformly convex soliton. By the soliton classification result in \cite{BCD}, we obtain that the limit is a round sphere. This implies the convergence of the normalized flow in $\mathbb{N}^{n+1}(\kappa)$ for $\alpha>\frac{1}{n+2}$.

The rest of this paper are organized as follows: in Section 2, we will recall some basic facts which will be used later. In Section 3, we prove the flow (\ref{unor gene flow}) converges to a point in finite time $T^*>0$ for $\alpha>0$. As a corollary, we prove that the flow (\ref{unnorm flow on spaceform}) converges to a point at finite time $T^*$ for $\alpha>0$. In section 4, we obtain the modified monotone quantity and the a priori estimates of the normalized flow. In section 5,  we prove the normalized flow in $\mathbb{N}^{n+1}(\kappa)$  converges to a geodesic sphere centered at the extinct point $q_0$ for $\alpha>\frac{1}{n+2}$.

\vspace{.1in}

\section{Preminaries}
In this section, we present some basic facts about space forms and the stereographic type projections which will be used later.
 
Under the geodesic polar coordinates, the metric of $\mathbb{N}^{n+1}(\kappa)$ can be denoted as
\begin{equation}
	\bar{g}=d\rho^2+\phi^2(\rho)dz^2,
\end{equation} 
where $\phi(\rho)=\sin(\rho)$, $\rho\in [0,\pi)$ when $\kappa=1$; and $\phi(\rho)=\sinh(\rho)$ when $\kappa=-1$; and $dz^2$ is the standard induced metric on $\mathbb{S}^n$ in Euclidean space.

Let $\Omega$ be a convex body in $\mathbb{N}^{n+1}(\kappa)$. Suppose $\mathcal{M}=\partial \Omega$ is smooth and strictly convex, denote the metric and the unit outer normal of $\mathcal{M}$ by $g_{ij}$, and $\nu$ respectively. Let $h_{ij}$ be the second fundamental form of $\mathcal{M}$ with respect to $\nu$ and $u=\langle \phi\frac{\partial}{\partial\rho},\nu\rangle$ the support function of $\mathcal{M}$. Suppose $q\in \mathcal{M}$ and, there is an open subset $\mathcal{N}$ of $\mathcal{M}$ containing $q$ such that $\langle\frac{\partial}{\partial \rho},\nu\rangle$ is strictly positive or negative ( doesn't change sign ) in $\mathcal{N}$, then $\mathcal{N}$ can be represented as a radial graph locally. As a local radial graph, it is well-known (see e.g. \cite{Guan-Li}) that in $\mathcal{N}$
\begin{equation}\label{gij hij}
	\begin{aligned}
		&g_{ij}=\rho_i\rho_j+\phi^2\delta_{ij},\\
		&\nu=\frac{\sigma}{\omega}(\frac{\partial}{\partial \rho}-\frac{\nabla\rho}{\phi^2}),\\
		&h_{ij}=\sigma(\sqrt{\phi^2+|\nabla\rho|^2})^{-1}(-\phi\rho_{ij}+2\phi'\rho_i\rho_j+\phi^2\phi'\delta_{ij}),
	\end{aligned}
\end{equation} 
 where $\rho_i=\nabla_i\rho$ and $\nabla$ is the covariant derivative on $\mathbb{S}^n$ with respect to an orthonormal frame, $\omega=\frac{\sqrt{\phi^2+|\nabla\rho|^2}}{\phi}$, $\sigma=1$ when $\langle\frac{\partial}{\partial \rho},\nu\rangle>0$, and $\sigma=-1$ when $\langle\frac{\partial}{\partial\rho},\nu\rangle<0$. Therefore, the Gauss curvature of $\mathcal{N}$ is given by 
\begin{equation}\label{K}
	K=\frac{\det h_{ij}}{\det g_{ij}}=\frac{\sigma^n\det(-\phi\rho_{ij}+2\phi'\rho_i\rho_j+\phi^2\phi'\delta_{ij})}{(\phi^2+|\nabla\rho|^2)^{\frac{n+2}{2}}\phi^{2n-2}}.
\end{equation}

To investigate the flow (\ref{unnorm flow on spaceform}) in $\mathbb{N}^{n+1}(\kappa)$, we project it to the tangent plane of $\mathbb{N}^{n+1}(\kappa)$ at a certain point. The sterographic projection can be found in many references,  see e.g. \cite{Gerhardt book, GLW}. Here we describe it briefly for completeness. 

When $\kappa=1$, $\mathbb{N}^{n+1}(\kappa)=\mathbb{S}^{n+1}$. For any $p\in\mathbb{S}^{n+1}$, denote $\mathcal{H}(p)=\{z\in\mathbb{S}^{n+1}:d_{\mathbb{S}^{n+1}}(p,z)<\frac{\pi}{2}\}$ the open hemisphere centered at $p$.  We consider the projection $\pi_p$ of $\mathcal{H}(p)$ onto the tangent plane $L_p$ of $\mathbb{S}^{n+1}\subset\mathbb{R}^{n+2}$ at $p$, defined by
\begin{equation}
	\pi_p:z\in\mathcal{H}(p)\mapsto \frac{z}{ \langle z, p\rangle}\in L_p.
\end{equation}
For a strictly convex hypersurface $\mathcal{M}\subset\mathbb{S}^{n+1}$, it must enclose a convex body and is contained in a hemisphere $\mathcal{H}$, see, for example \cite{dW}. Suppose $\mathcal{H}=\mathcal{H}(p)$ is centered at $p$, we can use the above projection $\pi_p$ to project $\mathcal{M}$ onto $L_p$. 

When $\kappa=-1$, $\mathbb{N}^{n+1}(\kappa)=\mathbb{H}^{n+1}$ is the hyperbolic space. For any point $p\in\mathbb{H}^{n+1}$, we can consider $\mathbb{H}^{n+1}$ as a submanifold of $\mathbb{R}^{n+2}$ with vertex $p$,
\begin{align*}
 \mathbb{H}^{n+1}=\{(x_1,\dots,x_{n+1},x_{n+2})\in\mathbb{R}^{n+2}|x_{n+2}^2-\sum_{i=1}^{n+1}x_i^2=1,x_{n+2}>0\}, 	
 \end{align*}
 and 
 \begin{align*}
 	p=(0,\dots,0,1).
 \end{align*}
 Let $L_p$ be the tangent plane of $\mathbb{H}^{n+1}$ at $p$, we define $\pi_p$ as 
 \begin{equation}
 	\pi_p:z\in\mathbb{H}^{n+1}\mapsto \frac{z}{\langle z, p\rangle}\in L_p.
 \end{equation}
 Note that if $z=(x_1,\dots,x_{n+1},x_{n+2})\in\mathbb{H}^{n+1}$, then $\pi_p(z)=(\frac{x_1}{x_{n+2}},\dots,\frac{x_{n+1}}{x_{n+2}},1)$. Thus, $\pi_p(\mathbb{H}^{n+1})$ is contained in the unit ball $B^{n+1}(p,1)$ of $L_p$ centered at $p$.
\begin{lemma}\label{KK hat rel}
	 Let $\mathcal{M}$ be a closed smooth strictly convex hypersurface in $\mathcal{N}^{n+1}(\kappa)$, $\Omega$ be the set enclosed by $\mathcal{M}$. Let $\pi:=\pi_p$ be defined as above and $\hat{\Omega}=\pi(\Omega)\subset L_p$ be the image of $\Omega$ under the projection, $\hat{u}:\mathbb{S}^n\to\mathbb{R}$, $x\mapsto\sup\{\langle y, x\rangle:y\in\hat{\Omega}\}$ be the support function of $\hat{\Omega}$, and $\hat{K}$ be the Gauss curvature of $\hat{\mathcal{M}}:=\partial\hat{\Omega}$. Then
	\begin{equation}\label{KKhat rela}
		K(q)=\Big(\frac{1+\kappa(\hat{u}^2+|\nabla \hat{u}|^2)}{1+\kappa\hat{u}^2}\Big)^{\frac{n+2}{2}}\hat{K}(\pi(q) ),\qquad\forall q\in\mathcal{M}.
	\end{equation} 
	where $x\in\mathbb{S}^n$ is the unique point such that $\pi(q)=\nabla u+ux$.
\end{lemma}
\begin{proof} 
	For the case $\kappa=1$, \cite{GLW} gives a detailed proof. We give a proof for the general case of space forms.  
	
	We identify the tangent plane $L_p$ with $\mathbb{R}^{n+1}$ and choose $p$ as the origin of $\mathbb{R}^{n+1}$. Recall that $\rho(q)$ is the geodesic distance from $q$ to the origin, let $r$ be the Euclidean distance from the origin. It is easy to see that
	\begin{equation}\label{r phi}
		r=\frac{\phi}{\phi'}=\left\{
		\begin{split}
			&\tan(\rho),\quad \kappa=1;\\
			&\tanh(\rho),\quad \kappa=-1.
		\end{split}
		\right.
	\end{equation}
	Since $\mathcal{M}$ is  strictly convex, we claim that the co-dimension of the set $\mathcal{Z}=\{q'\in\mathcal {M}|u(q')=0\}$ is one. In fact, let $\nabla^g$ denote the gradient of $\mathcal{M}$, $\Phi(\rho)=\int_0^\rho\phi(s)ds$, then $\nabla^{g}_{i}u=h^j_i \nabla^g_j\Phi$ (see e.g. \cite{Guan-Li}). If $\nabla^gu=0$, then $\nabla^g\Phi=0$. On the other hand, $u^2+|\nabla^g \Phi|^2=\phi^2$, so $\{q'\in\mathcal {Z}|\nabla^gu(q')=0\}=\{q'\in\mathcal {M}|\phi(q')=0\}$, i.e. a single point. If $\nabla^gu\neq0$, then the set $\{q'\in\mathcal {Z}|\nabla^gu(q')\neq0\}$ is co-dimension one by the implicit function theorem. For $q\in \mathcal{M}\setminus\mathcal{Z}$, there is a neighbourhood $\mathcal{N}$ of $q$ in $\mathcal{M}$ such that $u>0\,(<0)$ in $\mathcal{N}$, and $\mathcal{N}$ can be represented as local radial graph. Moreover, $\hat{u}>0\,(<0)$ in $\hat{\mathcal{N}}=\pi(\mathcal{N})$ if and only if $u>0\,(<0)$ in $\mathcal{N}$, and $\hat{\mathcal{N}}$ can be represented as a radial graph in the polar coordinates of $\mathbb{R}^{n+1}$ with origin $p$.

	Similar to (\ref{gij hij}) and (\ref{K}), we have in $\hat{\mathcal{N}}$ that 
	\begin{equation}\label{ghatij hhatij}
	\begin{aligned}
		&\hat{g}_{ij}=r_ir_j+r^2\delta_{ij},\\
		&\hat{\nu}=\frac{\sigma r}{\sqrt{r^2+|\nabla r|^2}}(\frac{\partial }{\partial r}-\frac{\nabla r}{r^2}),\\
		&\hat{h}_{ij}=\sigma(\sqrt{r^2+|\nabla r|^2})^{-1}(-rr_{ij}+2r_ir_j+r^2\delta_{ij}),
	\end{aligned}
\end{equation} 
\begin{equation}\label{Khat}
	\hat{K}=\frac{\det \hat{h}_{ij}}{\det \hat{g}_{ij}}=\frac{\sigma ^n\det(-rr_{ij}+2r_ir_j+r^2\delta_{ij})}{(r^2+|\nabla r|^2)^{\frac{n+2}{2}}r^{2n-2}},
\end{equation}
 where $\hat{g}_{ij}, \hat{\nu},\hat{h}_{ij}$ and $\hat{K}$ are the metric, the unit outer normal, the second fundamental form, and the Gauss curvature of $\hat{\mathcal{M}}$ respectively.
 
On the other hand, by (\ref{r phi}) and the fact that $\phi'^2-\phi\phi''=1$, we get
	\begin{equation}
	\begin{aligned}
		K=\frac{\sigma ^n\det(-rr_{ij}+2r_ir_j+r^2\delta_{ij})}{r^{2n-2}(r^2+\phi'^2|\nabla r|^2)^{\frac{n+2}{2}}}.
	\end{aligned}	
	\end{equation}
	Moreover, since $\phi(\rho)=\sinh(\rho)$ for $\kappa=-1$, $\phi(\rho)=\sin(\rho)$ for $\kappa=1$, thus $r^2=\frac{\phi^2}{\phi'^2}=\kappa(\frac{1}{\phi'^2}-1)$, i.e.
	\begin{equation}
		\phi'^2=\frac{1}{1+\kappa r^2}.
	\end{equation}
	Comparing this to (\ref{Khat}), we get
		\begin{equation}\label{KKhat 1}
		\frac{K}{\hat{K}}=\Big(\frac{r^2+|\nabla r|^2}{r^2+\frac{|\nabla r|^2}{1+\kappa r^2}}\Big)^{\frac{n+2}{2}}.
	\end{equation}
	It is well known that for $x\in\mathbb{S}^n$, $x$ is an unit outer normal of hypersurface defined by $\nabla \hat{u}+\hat{u}x\in\hat{\mathcal{N}}$, thus
	\begin{equation}
		r^2=\hat{u}^2+|\nabla \hat{u}|^2,
	\end{equation}
	\begin{equation}
		\hat{u}(x)=\langle r\frac{\partial}{\partial r}, \hat{\nu}\rangle=\frac{\sigma r^2}{\sqrt{r^2+|\nabla r|^2}}.
	\end{equation}
	Plugging the above two equations into (\ref{KKhat 1}), we get
	\begin{equation}
		\frac{K}{\hat{K}}=\Big(\frac{1+\kappa(\hat{u}^2+|\nabla \hat{u}|^2)}{1+\kappa\hat{u}^2}\Big)^{\frac{n+2}{2}}.
	\end{equation} 
	This proves (\ref{KKhat rela}) for $q\in\mathcal{M}\setminus\mathcal{Z}$. Since $\mathcal{M}\setminus\mathcal{Z}$ is dense in $\mathcal{M}$,  (\ref{KKhat rela}) holds for $q\in\mathcal{M}$.	   
\end{proof}

Let $\tilde{X}(\tau)$ be a family of hypersurfaces evolving by the flow (\ref{unnorm flow on spaceform}). Suppose we can represent $\tilde{X}(\tau)$ as $\{(\tilde{\rho}(z,t)z,z)\}$ as a radial graph locally over $\mathbb{S}^n$ in the polar coordinates with center $p$, then we obtain the scalar curvature flow equation (locally)
\begin{equation}\label{uno rho space}
	\tilde{\rho}_t=-\tilde{K}^\alpha \tilde{\omega},
\end{equation}
where $\tilde{\omega}=\frac{\sqrt{\phi(\tilde{\rho})^2+|\nabla\tilde{\rho}|^2}}{\phi(\tilde{\rho})}$.

Suppose that (\ref{unnorm flow on spaceform}) exists on $[0,T^*)$. Then we can project $\tilde{X}(\tau)$ into $L_{p_0}$ through the projection $\pi_{p_0}$, where $p_0$ is the outer center of $\tilde{X}(0)$.
 
\begin{lemma}\label{lemma proj}
  The image $\hat{\tilde{X}}(\tau):=\pi_{p_0}(\tilde{X}(\tau))$ evolves by 
\begin{equation}\label{unor X rn+1 evo}
	\hat{\tilde{X}}_\tau=-(1+\kappa|\hat{\tilde{X}}|^2)^{\frac{n+2}{2}\alpha+\frac{1}{2}}(1+\kappa\langle\hat{\tilde{X}},\hat{\nu}\rangle^2)^{-\frac{n+2}{2}\alpha+\frac{1}{2}}\hat{\tilde{K}}^\alpha(x,\tau)\hat{\nu},\qquad\tau\in[0,T^*),
\end{equation} 
where $\hat{\tilde{K}}$ and $\hat{\nu}$ are the Gauss curvature and the unit outer normal of $\hat{\tilde{X}}(\tau)$ respectively. The support function satisfies
\begin{equation}\label{rn+1 unor support}
\begin{aligned}
	\hat{\tilde{u}}_\tau(x,\tau)
	=&-(1+\kappa(\hat{\tilde{u}}^2+|\nabla\hat{\tilde{u}}|^2))^{\frac{n+2}{2}\alpha+\frac{1}{2}}(1+\kappa\hat{\tilde{u}}^2)^{-\frac{n+2}{2}\alpha+\frac{1}{2}}\hat{\tilde{K}}^\alpha(x,\tau),\qquad\tau\in [0,T^*).
\end{aligned}	
\end{equation}	
\end{lemma}
\begin{proof}
Since $\tilde{X}_0$ is strictly convex, $\tilde{X}(\tau)$ will stay strictly convex on a short time interval $[0,\delta)$. Thus, there is a set $\mathcal{Z}(\tau)\subset\tilde{X}(\tau)$ of measure zero, such that $\tilde{X}(\tau)$ can be represented as local radial graph and (\ref{uno rho space}) holds on $\tilde{X}(\tau)\setminus\mathcal{Z}(\tau)$. Plugging (\ref{r phi}) into (\ref{uno rho space}), we get the evolution equation of $\tilde{r}$ 
\begin{equation}\label{rn+1 unor radi}
	\tilde{r}_\tau=-\sqrt{(1+\kappa\tilde{r}^{2})^2+\frac{|\nabla \tilde{r}|^{2}}{\tilde{r}^2}(1+\kappa\tilde{r}^2)}\tilde{{K}}^\alpha
\end{equation}
holds in $\tilde{X}(\tau)\setminus{\mathcal{Z}}(\tau)$ for $\tau\in[0,\delta)$.

Note that $\frac{\hat{\tilde{u}}(x,\tau)_\tau}{\hat{\tilde{u}}}=\frac{\tilde{r}(z,\tau)_\tau}{\tilde{r}}$, by Lemma \ref{KK hat rel}, we obtain 
\begin{align*}
	\hat{\tilde{u}}_\tau(x,\tau)=&-(1+\kappa(\hat{\tilde{u}}^2+|\nabla \hat{\tilde{u}}|^2))^{\frac{n+2}{2}\alpha+\frac{1}{2}}(1+\kappa\hat{\tilde{u}}^2)^{-\frac{n+2}{2}\alpha+\frac{1}{2}}\hat{\tilde{K}}^\alpha(x,\tau)
	\end{align*}
holds in $\tilde{X}(\tau)\setminus{\mathcal{Z}}(\tau)$ for $\tau\in[0,\delta)$. This is the evolution equation of the support funtion when the hypersurfaces evolve by (\ref{unor X rn+1 evo}). By applying Lemma \ref{unor ki lower boud} with $T=\delta$, we obtain the principal curvatures of $\hat{\tilde{X}}$ have a  uniform positive lower bound $\varepsilon_0$ depending only on $n,\alpha,\tilde{X}_0$. This implies that $\tilde{X}(\tau)$ is uniformly convex on $[0,\delta]$. Then repeating this process, (\ref{unor X rn+1 evo}) and (\ref{rn+1 unor support}) hold on $\tilde{X}(\tau)\setminus\mathcal{Z}(\tau)$ for any $\tau\in[0,T^*)$. Since $\mathcal{Z}(\tau)$ is of measure zero for any fixed $\tau\in[0,T^*)$, this finishes the proof.       
\end{proof}

\section{Convergence to a point}
In this section, we prove that the flow (\ref{unnorm flow on spaceform}) converges to a point in finite time $T^*>0$. This is proved by proving the image flow of its projection in $\mathbb{R}^{n+1}$ converges to a point at $T^*$. More generally, we prove  the following theorem. 
\begin{theorem}\label{gene theorem}
	Suppose $\{\hat{\tilde{X}}(\tau)\}\subset\mathbb{R}^{n+1}$ is a family of hypersurfaces in $\mathbb{R}^{n+1}$ evolving by
	\begin{equation}\label{unor gene flow}
		\hat{\tilde{X}}_\tau=-\psi(\langle\hat{\tilde{X}},\hat{\nu}\rangle,x,\nabla \langle\hat{\tilde{X}},\hat{\nu}\rangle)\hat{\tilde{K}}^\alpha(x,\tau)\hat{\nu},
	\end{equation}
	with $\hat{\tilde{X}}(0)=\hat{\tilde{X}}_0$ strictly convex, where $\hat{\tilde{K}},\hat{\nu}$ are the Gauss curvature, and unit outer normal of $\hat{\tilde{X}}$ respectively, $\alpha>0$ is a positive constant, $\psi:(\mathbb{R}\times T\mathbb{S}^n)\to\mathbb{R}^n$ is a smooth function satisfying
	\begin{equation}\label{gene psi cond}
	\begin{aligned}
		  \frac{1}{A}\leq \psi\leq A,\\
		  \|\psi\|_{C^2}\leq A,
		 \end{aligned}
		 	\end{equation}
	for some positive constant $A>0$ as long as the flow (\ref{unor gene flow}) exists. Then the flow (\ref{unor gene flow}) converges to a point in finite time $T^*>0$ with $T^*$ depending only on $n,\alpha,\hat{\tilde{X}}_0$ and $A$.
\end{theorem}
When we consider the flow (\ref{unor X rn+1 evo}), $\psi=(1+\kappa|\hat{\tilde{X}}|^2)^{\frac{n+2}{2}\alpha+\frac{1}{2}}(1+\kappa\langle\hat{\tilde{X}},\hat{\nu}\rangle^2)^{-\frac{n+2}{2}\alpha+\frac{1}{2}}$ is uniformly bounded from below, since the flow is contracting.
\begin{corollary}\label{coro to point}
	The flow (\ref{unnorm flow on spaceform}) converges to a point in finite time $T^*$ with $T^*$ depending only on $\tilde{X}_0,n,\alpha$.
\end{corollary}

\begin{proof}
Let $p_0$ be the outer center of $\tilde{X_0}$, we consider the projection $\pi_{p_0}$ of $\tilde{X}(\tau)$ into $L_{p_0}$ on $[0,T^*)$. By Lemma \ref{lemma proj}, the image $\hat{\tilde{X}}(\tau)=\pi_{p_0}(\tilde{X}(\tau))$ will evolve by (\ref{unor X rn+1 evo}), which is a special case of (\ref{unor gene flow}) with $\psi=(1+\kappa|\hat{\tilde{X}}|^2)^{\frac{n+2}{2}\alpha+\frac{1}{2}}(1+\kappa\langle\hat{\tilde{X}},\hat{\nu}\rangle^2)^{-\frac{n+2}{2}\alpha+\frac{1}{2}}$. We can check that $\hat{\tilde{X}}_0$ is strictly convex and $\psi$ satisfies (\ref{gene psi cond}). By Theorem \ref{gene theorem}, $\hat{\tilde{X}}(\tau)$ converges to a point $\hat{q}$ in finite time $T^*>0$. Thus, $\pi_{p_0}^{-1}(\hat{\tilde{X}}(\tau))$ converges to the point $\pi_{p_0}^{-1}(q)$. 
\end{proof}
Next, we prove Theorem \ref{gene theorem}. 
Note that under the flow (\ref{unor gene flow}), the support function $\hat{\tilde{u}}$ evolves by
\begin{equation}\label{unor gene u}
	\hat{\tilde{u}}_\tau=-\psi(\hat{\tilde{u}},x,\nabla \hat{\tilde{u}})\hat{\tilde{K}}^\alpha(x,\tau).
\end{equation} 
  
 \begin{lemma}\label{unor ki lower boud}
 	 Suppose $\psi$ satisfies (\ref{gene psi cond}), then under the flow (\ref{unor gene flow}), there exists a constant $\varepsilon_0=\varepsilon_0(n,\alpha,A, \hat{\tilde{X}}_0)>0$ such that the principal curvatures $\hat{\tilde{\kappa}}_i$ of $\hat{\tilde{X}}$ satisfies
 	\begin{equation}
 		\hat{\tilde{\kappa}}_i\geq \varepsilon_0,\qquad \tau\in[0,T],
 	\end{equation}
 	for any $T<T^*$ fixed, where $T^*$ is the maximal existence time of (\ref{unor gene flow}).
 \end{lemma}
\begin{proof}
In the following of the proof, we omit $\hat{}$ and $\tilde{}$ and use $t$ instead of $\tau$  for simplicity.

	Let $(W_{ij})=(u_{ij}+u\delta_{ij})$ be the inverse of the Weingarten matrix of $X$, whose eigenvalues $(\lambda_1,\dots,\lambda_n)$ are the principal radii of curvature of $X$. To prove the lower bound of $\kappa_i$, it suffices to prove the upper bound of $\lambda(x,t):=\max_{i=1,\dots,n}\lambda_i(x,t)$. We consider the function $\bar{G}(x,t):=\log \lambda+\frac{L}{2}r^2$, where $L>0$ is a large constant to be determined. Suppose the maximum of $\bar{G}$ on $\mathbb{S}^n\times [0,T]$ is attained at $(x_0,t_0)$,  we choose a local orthonormal frame $e_1,\dots,e_n$ around $x_0$ such that $\{W_{ij}(x,t)\}$ is diagonal at $(x_0,t_0)$ and $W_{11}(x_0,t_0)=\lambda_1(x_0,t_0)=\lambda(x_0,t_0)$. Then the function 
	\begin{equation}
		G(x,t):=\log W_{11}(x,t)+\frac{L}{2}r^2
	\end{equation} 
	also attains its maximum at $(x_0,t_0)$. Let $(W^{ij})=(W_{ij})^{-1}$ be the inverse of $(W_{ij})$, $F^{pq}:=\alpha \psi K^{\alpha}W^{pq}$, then at $(x_0,t_0)$
		\begin{align*}
			W_{11,t}=&(u_{11}+u)_t\\
			=&-K^\alpha[\psi+\psi_uu_{11}+\psi_{u_i}u_{i11}+\psi_{x_1x_1}+\psi_{uu}u_1^2+\psi_{u_1u_1}u_{11}^2+2\psi_{x_1u}u_1\\
			&+2\psi_{x_1u_1}u_{11}+2\psi_{uu_1}u_{11}u_1-2\alpha W^{ii}W_{ii1}(\psi_{x_1}+\psi_uu_1+\psi_{u_1}u_{11})\\
			&+\alpha^2\psi(W^{ii}W_{ii1})^2+\alpha \psi W^{ii}W^{jj}W_{ij1}^2-\alpha \psi W^{ii}W_{ii,11}].\\
       \end{align*}
Using the formula for commutating covariant derivatives on $\mathbb{S}^n$, we have at $(x_0,t_0)$, 
		\begin{align*}
			&W_{i11}=W_{11i},\\
			&W_{ii,11}=W_{11,ii}+W_{ii}-W_{11},
		\end{align*}
	Thus
	\begin{equation}\label{unor W11 evo}
	    \begin{aligned}
			&\quad W_{11,t}-F^{pp}W_{11,pp}\\
			=&-K^\alpha[\psi+\psi_uW_{11}-\psi_uu+\psi_{u_i}W_{11i}-\psi_{u_1}u_1+\psi_{x_1x_1}+\psi_{uu}u_1^2+\psi_{u_1u_1}W_{11}^2\\
			&-2\psi_{u_1u_1}W_{11}u+\psi_{u_1u_1}u^2+2\psi_{x_1u}u_1
			+2\psi_{x_1u_1}W_{11}-2\psi_{x_1u_1}u\\
			&+2\psi_{uu_1}W_{11}u_1
			-2\psi_{uu_1}uu_1-2\alpha W^{ii}W_{ii1}(\psi_{x_1}+\psi_uu_1+\psi_{u_1}W_{11}\\
			&-\psi_{u_1}u)+\alpha^2\psi(W^{ii}W_{ii1})^2+\alpha \psi W^{ii}W^{jj}W_{ij1}^2+\alpha \psi W^{ii}(W_{11}-W_{ii})].\\
	 \end{aligned}
	 \end{equation}
On the other hand, we have at $(x_0,t_0)$ 
	\begin{align*}
		&r^2_t=2uu_t+2u_iu_{it}=-2K^\alpha[u\psi+\psi_{x_i}u_i+\psi_u|\nabla u|^2+\psi_{u_i}u_{ii}u_i-\alpha \psi W^{pp}W_{ppi}u_i],\\
		&(r^2)_{pp}=2(W_{pi}u_{i})_p=2W_{ppi}u_i+2W_{pp}^2-2uW_{pp},
	\end{align*}
	which implies
	\begin{equation}
	\begin{aligned}
		&r^2_t-F^{pp}(r^2)_{pp}\\
		=&-2K^\alpha[u\psi+\psi_{x_i}u_i+\psi_u|\nabla u|^2+\psi_{u_i}u_{ii}u_i+\alpha \psi W^{pp}(W_{pp}^2-uW_{pp})]\\
		=&-2K^\alpha[(-n\alpha+1)u\psi+\psi_{x_i}u_i+\psi_u|\nabla u|^2+\psi_{u_i}u_{ii}u_i+\alpha \psi W_{pp}].\\
	\end{aligned}
	\end{equation}
	By maximum principle, at $(x_0,t_0)$, we have
	\begin{equation}\label{w11 criti}
	W_{11i}=-\frac{L}{2}W_{11}(r^2)_i=-LW_{11}W_{ii}u_i,
	\end{equation}
	and
	\begin{align*}
		 0\leq& G_t-F^{pp}G_{pp}\\
		 =&\frac{W_{11t}-F^{pp}W_{11,pp}}{W_{11}}+F^{pp}\frac{W_{11p}^2}{W_{11}^2}+\frac{L}{2}(r^2_t-F^{pp}(r^2)_{pp})\\
		 =&-K^\alpha[\psi_{u_1u_1}W_{11}+\psi_{u_i}\frac{W_{11i}}{W_{11}}+\psi_{u}-2\psi_{u_1u_1}u+2\psi_{x_1u_1}+2\psi_{uu_1}u_1\\
		 &+\frac{1}{W_{11}}(\psi-\psi_uu-\psi_{u_1}u_1+\psi_{x_1x_1}+\psi_{uu}u_1^2+\psi_{u_1u_1}u^2+2\psi_{x_1u}u_1-2\psi_{x_1u_1}u-2\psi_{uu_1}uu_1)\\
		 &-2\alpha W^{ii}W_{ii1}(\psi_{u_1}+\frac{\psi_{x_1}+\psi_uu_1-\psi_{u_1}u}{W_{11}})+\alpha^2\frac{\psi(W^{ii}W_{ii1})^2}{W_{11}}+\alpha\frac{\psi W^{ii}W^{jj}W_{ij1}^2}{W_{11}}\\
		 &+\alpha\psi W^{ii}(1-\frac{W_{ii}}{W_{11}})-\alpha\psi W^{pp}\frac{W_{11p}^2}{W_{11}^2}+(1-n\alpha)Lu\psi+L\psi_{x_i}u_i+L\psi_u|\nabla u|^2\\
		 &+L\psi_{u_i}W_{ii}u_i-L\psi_{u_i}uu_i+\alpha L\psi W_{pp}]\\
		 \leq &-K^{\alpha}[-C(n,\alpha,A,L,X_0)+\psi_{u_1u_1}W_{11}-L\psi_{u_i}W_{ii}u_i-2\alpha W^{ii}W_{ii1}(\psi_{u_1}+\frac{\psi_{x_1}+\psi_uu_1-\psi_{u_1}u}{W_{11}})\\
		 &+\alpha^2\frac{\psi(W^{ii}W_{ii1})^2}{W_{11}}+\alpha\frac{\psi W^{ii}W^{jj}W_{ij1}^2}{W_{11}}+\alpha\psi W^{ii}(1-\frac{W_{ii}}{W_{11}})-\alpha\psi W^{pp}\frac{W_{11p}^2}{W_{11}^2}\\
		 &+L\psi_{u_i}W_{ii}u_i+\alpha L\psi W_{pp}]\\
		 \end{align*}
 where we use (\ref{w11 criti}), (\ref{gene psi cond}) and the fact that $u^2\leq u^2+|\nabla u|^2\leq C(X_0)$ in the last step since the flow is contracting.  By Cauchy Schwartz inequality,		 
\begin{equation}
\begin{aligned}
   &G_t-F^{pp}G_{pp}\\
    \leq &-K^{\alpha}[-C(n,\alpha,A,L,X_0)+(\psi_{u_1u_1}-C(n,\alpha,A,X_0))W_{11}\\
		 &+\alpha\frac{\psi W^{ii}W^{jj}W_{ij1}^2}{W_{11}}-\alpha\psi W^{pp}\frac{W_{11p}^2}{W_{11}^2}+\alpha L\psi W_{pp}]\\
	\leq &-K^{\alpha}[-C(n,\alpha,A,L,X_0)+(\psi_{u_1u_1}-C(n,\alpha,A,X_0))W_{11}\\
		 &+\alpha\frac{\psi W^{ii}W^{11}W_{11i}^2}{W_{11}}-\alpha\psi W^{pp}\frac{W_{11p}^2}{W_{11}^2}+\alpha L\psi W_{pp}]\\	
		 \leq &	-K^{\alpha}[-C(n,\alpha,A,L,X_0)-C(n,\alpha,A,X_0)W_{11}+\alpha L\psi W_{11}].	
\end{aligned} 
\end{equation}
Since $\psi\geq \frac{1}{A}>0$ is bounded from below by (\ref{gene psi cond}), we can take $L\geq C_1(n,\alpha,A,X_0)$ sufficiently large such that $W_{11}\leq C_2(n,\alpha,A,X_0)$. This implies that
	 \begin{equation}
		\lambda(x,t)\leq C_2(n,\alpha,A,X_0).
		\end{equation}  
\end{proof}

\begin{lemma}\label{unno K above}
For any smooth, strictly convex solution $\{\hat{\tilde{X}}(\tau)\}_{[0,T]}$ of equation $(\ref{unor gene flow})$ with $R_-\le r_-(\hat{\tilde{X}}(\tau))\le r_+(\hat{\tilde{X}}(\tau))\le R_+$ for $\tau\in [0,T]$, 
\begin{equation}
\max_{\tau\le T,x\in \mathbb{S}^n} \psi \hat{\tilde{K}}^{\alpha}\le \max\{ C(\alpha,n,\hat{\tilde{X}}_0,A)\frac{R_+}{R_-^{1+n\alpha}}, \max_{x\in \mathbb{S}^n}\psi \hat{\tilde{K}}^{\alpha}(x,0)R_+\}.
\end{equation}

\end{lemma}
\begin{proof}
In the following, we omit $\hat{}$ and $\tilde{}$ and use $t$ instead of $\tau$  for simplicity.
Let $\sigma_n$ be the $n$-th elementary symmetric function with $\sigma_{n}(W)=K^{-1}$. From the evolution equation (\ref{unor gene u}), we obtain
\begin{equation}
\frac{\partial}{\partial t}W_{kl}=-((\psi\sigma_n^{-\alpha})_{kl}+\delta_{kl}\psi \sigma_n^{-\alpha}).
\end{equation}
There exist a point $z_0\in \Int(\Omega(T))$ such that $u_{z_0}(x,T):=u(x,T)-\langle z_0, x\rangle >\frac{R_-}{2}$. Setting $q=\frac{\psi \sigma_n^{-\alpha}}{u_{z_0}-\frac{R_-}{2}}$, this implies 
\begin{align*}
\frac{\partial}{\partial t}(\psi\sigma_n^{-\alpha})&=\alpha \psi(u_{z_0}-\frac{R_-}{2})\sigma_n^{-(\alpha+1)}\sigma_n^{kl} q_{kl}+2\alpha \psi \sigma_n^{-(\alpha+1)}\sigma_n^{kl}q_k(u_{z_0})_l\\
&+\alpha \psi \sigma_n^{-(\alpha+1)}\sigma_n^{kl}((u_{z_0})_{kl}+\delta^l_k u_{z_0})q- \frac{\alpha R_-}{2} \psi \sigma_n^{-(\alpha+1)}Hq+\sigma_n^{-\alpha}\psi_t,
\end{align*}
where $H$ is the mean curvature. Since $M_t$ is shrinking we have $\Omega(T)\subset \Omega(t)$ for any $t<T$, then we obtain $u_{z_0}(x,t)>\frac{R_-}{2}$ for all $(x,t)\in \mathbb{S}^n\times [0,T]$. Since $W_{kl}= u_{kl}+\delta^l_k u=( u_{z_0})_{kl}+\delta^l_k u_{z_0}$, we obtain 
\begin{align*}
\frac{\partial}{\partial t}q&=-\frac{u_{tt}}{u_{z_0}-\frac{R_-}{2}}+\frac{u_t(u_{z_0})_{t}}{(u_{z_0}-\frac{R_-}{2})^2}\\
&=-\frac{(-\psi\sigma_n^{-\alpha})_{t}}{u_{z_0}-\frac{R_-}{2n}}+q^2\\
&=\alpha \psi\sigma_n^{-(\alpha+1)}\sigma_n^{kl}q_{kl}+\frac{2\alpha \psi \sigma_n^{-(\alpha+1)}\sigma_n^{kl}q_k(u_{z_0})_l}{u_{z_0}-\frac{R_-}{2}}\\
&+q^2(1+\alpha n-\frac{R_-}{2}\alpha H)+\frac{\sigma_n^{-\alpha}\psi_t}{u_{z_0}-\frac{R_-}{2}},
\end{align*}
where $H\ge n\sigma_n^{-\frac{1}{n}}$.

\begin{align*}
\frac{\partial}{\partial t}\psi=\psi_uq(\frac{R_-}{2}-u_{z_0})+\psi_{u_k}q_k(\frac{R_-}{2}-u_{z_0})-\psi_{u_k}q(u_{z_0})_k.
\end{align*}
At the maximum point $(x_0,t_0)$ of $q$, we obtain
\begin{align*}
 &q^2(1+\alpha n-\frac{\alpha R_- H}{2}+\frac{\psi_u}{\psi}(\frac{R_-}{2}-u_{z_0})-\frac{\psi_{u_k}(u_{z_0})_k}{\psi}\big)\\
\le& q^2(1+n\alpha-C(n,\alpha)R_-(\frac{u_{z_0}-\frac{R_-}{2}}{\psi})^{\frac{1}{n\alpha}}q^{\frac{1}{n\alpha}}+C(X_0,A))\\
\le& q^2(C(\alpha, n, X_0, A)-C(\alpha,n,A)R_-^{1+\frac{1}{n\alpha}}q^{\frac{1}{n\alpha}}).
\end{align*}
 Hence,
\[ \max_{t\le T,x\in \mathbb{S}^n}q(x,t)=q(x_0,t_0)\le \max\{ \frac{C(\alpha,n,X_0, A)}{R_-^{1+n\alpha}},\max_{x\in \mathbb{S}^n}q(x,0)\}.\]
That is,
\begin{align*}
\max_{t\le T,x\in \mathbb{S}^n} \psi \sigma_n^{-\alpha}&=\max_{t\le T,x\in \mathbb{S}^n} \Big(q(x,t)(u(x,t)-\frac{R_-}{2})\Big)\\
&\le \max\{ C(\alpha,n,X_0, A)\frac{R_+}{R_-^{1+n\alpha}},  \max_{x\in \mathbb{S}^n}\psi\sigma_n^{-\alpha}(x,0)R_+\}.
\end{align*}
\end{proof}

\begin{proof}[Proof of Theorem \ref{gene theorem}]
Since $\hat{\tilde{X}}(\tau)$ is a contracting flow, $|\hat{\tilde{u}}|_{C^{0}}\leq C(\hat{\tilde{X}}_0)$. Since $\hat{\tilde{X}}$ is strictly convex, the $C^0$ estimate implies the $C^1$ estimate of $\hat{\tilde{u}}$. By the definition of $T^*$, $r_+\geq\varepsilon>0$ on $[0,T]$ for any $T<T^*$. On the other hand, due to Lemma 2.2 in \cite{CW} and Lemma \ref{unor ki lower boud},  we have 
\begin{equation}\label{r+- rela}
r^2_+\le C(n,\alpha, A,\hat{\tilde{X}}_0)r_-,
\end{equation} 
as long as the flow exists. By (\ref{r+- rela}), Lemma \ref{unor ki lower boud} and Lemma \ref{unno K above} 
\begin{equation}
\|\hat{\tilde{u}}(x,\tau)\|_{C^2(\mathbb{S}^n\times [0,T])}\leq C(n,\alpha, A,\hat{\tilde{X}}_0, \varepsilon).	
\end{equation}
Since equation (\ref{unor gene u}) is a concave parabolic equation, by Krylov-Safanov's theorem and the standard theory on parabolic equations, this implies that $\hat{\tilde{X}}(\tau)$ is smooth on $[0,T]$. Thus, $\lim_{\tau\to T^*}r_-=\lim_{\tau\to T^*}r_+=0$. That is, $\hat{\tilde{X}}(\tau)$ converges to a point $\hat{q}$ as  $\tau\to T^*$. Set the initial value of the flow (\ref{unor gene flow}) to be the boundary of the outer ball of $\hat{\tilde{X}}_0$, then the solution will be a family of geodesic spheres with radii $\{\tilde{r}(\tau)\}$ satisfying the ODE
 \begin{equation}
 	\frac{\partial \tilde{r}}{\partial\tau}=-\frac{\psi}{\tilde{r}^{n\alpha}}\leq -\frac{1}{A\tilde{r}^{n\alpha}},
 	\end{equation}
 which converges to zero in finite time. By comparison principle, $\hat{\tilde{X}}(\tau)$ will converge to a point in finite time. 
 \end{proof}

\section{ the rescaled flow}
The un-normalized flow (\ref{unnorm flow on spaceform}) converges to a point $q_0\in \mathbb{N}^{n+1}(\kappa)$ as $\tau \rightarrow T^*$ by Corollary \ref{coro to point}. For $\kappa=1$ and sufficiently small $\delta_1$, $\tilde{X}(\tau)$ is contained in the open hemi-sphere $\mathcal{H}(q_0)$ when $\tau \in [T^*-\delta_1,T^*]$. We consider a new geodesic polar coordinate centered at $q_0$ and use the new projection $\pi_{q_0}$ of ${\tilde{X}(\tau)}$ for $\tau \in [T^*-\delta_1,T^*]$ onto the hyperplane of $\mathbb{R}^{n+2}$ which is tangent to $\mathbb{N}^{n+1}(\kappa)$ at $q_0$. We re-scale $\hat{\tilde{X}}(\tau)$ to keep the enclosed volume fixed.
Let $\hat{X}(x,t)=e^{t}\hat{\tilde{X}}(x,\tau)$, we have
\begin{equation}
\begin{aligned}
	\hat{X}_t(x,t)=&\hat{X}(x,t)+e^t\hat{\tilde{X}}_\tau(x,\tau)\tau'(t)\\
	=&\hat{X}(x,t)-e^{(n\alpha+1)t}\tau'(t)(1+\kappa|\hat{\tilde{X}}|^2)^{\frac{n+2}{2}\alpha+\frac{1}{2}}(1+\kappa|\langle \hat{\tilde{X}},\hat{\nu} \rangle|^2)^{-\frac{n+2}{2}\alpha+\frac{1}{2}}	\hat{K}^\alpha\\
	=&\hat{X}(x,t)-e^{(n\alpha+1)t}\tau'(t)(1+\kappa\frac{|\hat{X}|^2}{e^{2t}})^{\frac{n+2}{2}\alpha+\frac{1}{2}}(1+\kappa\frac{|\langle \hat{X},\hat{\nu} \rangle|^2}{e^{2t}})^{-\frac{n+2}{2}\alpha+\frac{1}{2}}	\hat{K}^\alpha,
	\end{aligned}
\end{equation}  
and the corresponding support function evolves by
\begin{equation}\label{evol-u}
	\hat{u}_t=\hat{u}-e^{(n\alpha+1)t}\tau'(t)(1+\kappa\frac{\hat{r}^2}{e^{2t}})^{\frac{n+2}{2}\alpha+\frac{1}{2}}(1+\kappa\frac{\hat{u}^2}{e^{2t}})^{-\frac{n+2}{2}\alpha+\frac{1}{2}}\hat{K}^\alpha.
\end{equation}
Since 
\begin{align*}
	\frac{d\Vol(\hat{\Omega})}{dt}=&\int_{\mathbb{S}^n}\hat{u}_t\sigma_nd\theta\\
	=&\int_{\mathbb{S}^n}\hat{u}\sigma_nd\theta-\tau'(t)e^{(n\alpha+1)t}\int_{\mathbb{S}^n}(1+\kappa\frac{\hat{r}^2}{e^{2t}})^{\frac{n+2}{2}\alpha+\frac{1}{2}}(1+\kappa\frac{\hat{u}^2}{e^{2t}})^{-\frac{n+2}{2}\alpha+\frac{1}{2}}\hat{K}^{\alpha-1}d\theta.
\end{align*}
We obtain 
\begin{equation}
\tau'(t)=\frac{\int_{\mathbb{S}^n}\hat{u}\sigma_nd\theta}{e^{(n\alpha+1)t}\int_{\mathbb{S}^n}(1+\kappa\frac{\hat{r}^2}{e^{2t}})^{\frac{n+2}{2}\alpha+\frac{1}{2}}(1+\kappa\frac{\hat{u}^2}{e^{2t}})^{-\frac{n+2}{2}\alpha+\frac{1}{2}}\hat{K}^{\alpha-1}d\theta}=\frac{(n+1)|\hat{\tilde{\Omega}}(\tau)|}{\frac{d|\hat{\tilde{\Omega}}(\tau)|}{d\tau}},
\end{equation}
that is 
\begin{equation}\label{tau,t}
t=\frac{1}{n+1}\log \Big(\frac{|B(1)|}{|\hat{\tilde{\Omega}}(\tau)|}\Big),
\end{equation} 
where $B(1)$ denotes the unit ball in $\mathbb{R}^{n+1}$.
Note that $|\hat{\tilde{\Omega}}(\tau)|$ approaches zero as $\tau$ approaches $T^*$, and consequently $t$ approaches infinity as $\tau$ approaches $T^*$ and the solution $\hat{X}(x,t)$ exists for all positive time.
We get
\begin{equation}\label{nor X equa eucl}
\begin{aligned}
	\hat{X}_t(x,t)=&\hat{X}(x,t)-\frac{(1+\kappa\frac{|\hat{X}|^2}{e^{2t}})^{\frac{n+2}{2}\alpha+\frac{1}{2}}(1+\kappa\frac{\langle\hat{X},\hat{\nu}\rangle^2}{e^{2t}})^{-\frac{n+2}{2}\alpha+\frac{1}{2}}\hat{K}^{\alpha}}{\dashint_{\mathbb{S}^n}(1+\kappa\frac{|\hat{X}|^2}{e^{2t}})^{\frac{n+2}{2}\alpha+\frac{1}{2}}(1+\kappa\frac{\langle\hat{X},\hat{\nu}\rangle^2}{e^{2t}})^{-\frac{n+2}{2}\alpha+\frac{1}{2}}\hat{K}^{\alpha-1}d\theta}\hat{\nu},
	\end{aligned}
\end{equation}
with initial value $\hat{X}_0=\Big(\frac{B(1)}{|\hat{\tilde{\Omega}}(T^*-\delta_1)|}\Big)^{\frac{1}{n+1}}\pi_{q_0}(\tilde{X}_{T^*-\delta_1})$. The support function satisfies    
\begin{equation}\label{norm u evo general}
\begin{aligned}
	\hat{u}_t&=\hat{u}-\frac{(1+\kappa\frac{\hat{r}^2}{e^{2t}})^{\frac{n+2}{2}\alpha+\frac{1}{2}}(1+\kappa\frac{\hat{u}^2}{e^{2t}})^{-\frac{n+2}{2}\alpha+\frac{1}{2}}\hat{K}^{\alpha}}{\dashint_{\mathbb{S}^n}(1+\kappa\frac{\hat{r}^2}{e^{2t}})^{\frac{n+2}{2}\alpha+\frac{1}{2}}(1+\kappa\frac{\hat{u}^2}{e^{2t}})^{-\frac{n+2}{2}\alpha+\frac{1}{2}}\hat{K}^{\alpha-1}d\theta}\\
	&=\hat{u}-\frac{\psi\hat{K}^\alpha}{\dashint_{\mathbb{S}^n}\psi\hat{K}^{\alpha-1}d\theta},\\
\end{aligned}
\end{equation}
with initial value $\hat{u}_0=\langle\hat{X}_0, \hat{\nu}\rangle$, $\psi=(1+\kappa\frac{\hat{r}^2}{e^{2t}})^{\frac{n+2}{2}\alpha+\frac{1}{2}}(1+\kappa\frac{\hat{u}^2}{e^{2t}})^{-\frac{n+2}{2}\alpha+\frac{1}{2}}$.

\subsection{Entropy}
 \cite{AGN} defined an 'entropy'  functional by
\begin{equation}
	\mathcal{E}_\alpha(\Omega):=\sup_{z_0\in\Omega}\mathcal{E}_\alpha(\Omega, z_0),
\end{equation}
where
\begin{equation}\label{entropy functional}
	\mathcal{E}_{\alpha}(\Omega,z_0)=\left\{\begin{split}
		&\frac{\alpha}{\alpha-1}(\log{\dashint_{\mathbb{S}^n}u_{z_0}^{1-\frac{1}{\alpha}}}d\theta),\quad& \alpha\neq 1;\\
		&\dashint_{\mathbb{S}^n}\log u_{z_0}d\theta, \quad & \alpha=1,
	\end{split}
	\right.
\end{equation}
where $u_{z_0}(x):=\sup_{z\in\Omega}\langle z-z_0,x\rangle$ is the support function in direction $x$ with respect to $z_0$. By Lemma 2.5 in \cite{AGN}, there exists a unique point $z_{e}\in\Int(\Omega)$ such that $\mathcal{E}_\alpha(\Omega)=\mathcal{E}_{\alpha}(\Omega,z_e)$ if $\Omega$ is a bounded convex domain with $\Int(\Omega)\neq\emptyset$. $z_e$ is called the 'entropy point' of $\Omega$.     

Denote $\hat{\Omega}_{t}=\hat{\Omega}(t)$ and $\hat{\tilde{\Omega}}_{\tau}=\hat{\tilde{\Omega}}(\tau)$. We remark that entropy of $\hat{X}$ is related to $\hat{\tilde{X}}$ as follows:
\begin{equation}
	\mathcal{E}_{\alpha}(\hat{\Omega}_t,e^tz_0)=\mathcal{E}_{\alpha}(\hat{\tilde{\Omega}}_{\tau},z_0)-\frac{1}{n+1}\log \Big(\frac{|\hat{\tilde{\Omega}}_{\tau}|}{|B(1)|}\Big).
\end{equation}
Note that $\psi=(1+\kappa\frac{\hat{r}^2}{e^{2t}})^{\frac{n+2}{2}\alpha+\frac{1}{2}}(1+\kappa\frac{\hat{u}^2}{e^{2t}})^{-\frac{n+2}{2}\alpha+\frac{1}{2}}$ is uniformly bounded from the contracting result in Section 3.

\subsection{Monotonicity inequality}
\begin{theorem}\label{monotone quantity}
	Under the normalized flow (\ref{norm u evo general}), $\mathcal{E}_{\alpha}(\hat{\Omega}_t)+C(n,\alpha,\tilde{X}_0)e^{-\frac{2(n+1)}{2n+1}t}$ is non-increasing when $t\ge t_0(n,\alpha,\tilde{X}_0,T^*-\delta_1)$ sufficiently large. Furthermore, $\mathcal{E}^{\infty}_{\alpha}:=\lim\limits_{t \to \infty}\big(\mathcal{E}(\hat{\Omega}_t)+C(n,\alpha,\tilde{X}_0)e^{-\frac{2(n+1)}{2n+1}t}\big)$ exists if $\alpha \ge \frac{1}{n+2}$. 
	\end{theorem}	

\begin{proof}
By (\ref{tau,t}),  we obtain $|\hat{\tilde{\Omega}}_{\tau}|=|B(1)|e^{-(n+1)t}$. By (\ref{r+- rela}) and the relation 
\begin{equation}\label{vol and r}
	C(n)r_-^{n-1}r_+\leq\Vol(\hat{\tilde{\Omega}}_\tau)\leq C(n)r_-r_+^{n-1},
\end{equation}
we have 
     \begin{equation}\label{decayest}
    C(n,\alpha,\tilde{X}_0)e^{-\frac{2(n+1)}{n+2}t}\le r_-\le r_+\le C(n,\alpha,\tilde{X}_0)e^{-\frac{n+1}{2n+1}t}
     \end{equation}
for some positive constant $C(n,\alpha,\tilde{X}_0)$. 

Fix $t_0>0$ to be determined and for any $t_1\geq t_0$, $\mathcal{E}_\alpha(\hat{\Omega}_{t_1})=\mathcal{E}_\alpha(\Omega,z_{e(t_1)})$, where $u_{e(t_1)}$ is the support function with respect to the unique entropy point $z_e(t_1)\in \Int(\hat{\Omega})$. Then for $t<t_1$ but very close to $t_1$, one still has $\hat{u}_{e(t)}(x,t):=\hat{u}(x,t)-\langle \exp (t-t_1)z_e(t_1), x \rangle >0$. Now we adopt the argument in \cite{Guan-Ni} to deduce the monotonicity property of the entropy. 

Case 1: $\alpha =1$. 

\begin{align*}
\frac{d}{dt}\dashint_{\mathbb{S}^n}\log \hat{u}_{e(t)}(x,t)d\theta &=\dashint_{\mathbb{S}^n}\frac{\hat{u}_{e(t)}-\frac{\psi\hat{K}}{\dashint_{\mathbb{S}^n}\psi d\theta}}{\hat{u}_{e(t)}}d\theta\\
&\le 1-\dashint_{\mathbb{S}^n}\frac{\hat{K}}{\hat u_{e(t)}}(1-C(n,\tilde{X}_0)e^{-\frac{2(n+1)}{2n+1}t})d\theta \\
&=-\dashint_{\mathbb{S}^n}\big(\sqrt{\frac{\hat{u}_{e(t)}}{\hat{K}}}-\sqrt{\frac{\hat{K}}{\hat{u}_{e(t)}}}\big)^2\big(1-C(n,\tilde{X}_0)e^{-\frac{2(n+1)}{2n+1}t}\big)d\theta\\
&+C(n,\tilde{X}_0)e^{-\frac{2(n+1)}{2n+1}t}
\end{align*}
for some positive constant $C(n,\tilde{X}_0)$, where we use (\ref{decayest}) in the second step. Choose $t_0$ such that $1-C(n,\tilde{X}_0)e^{-\frac{2(n+1)}{2n+1}t_0}\ge \frac{1}{2}$ and $t_0\ge T^*-\delta_1$. This implies that there exists a $\delta>0$ such that for $t\in (t_1-\delta,t_1)$,
\begin{align*}
&\mathcal{E}(\hat{\Omega}_t)+C(n,\tilde{X}_0)e^{-\frac{2(n+1)}{2n+1}t}\\
\ge& \dashint_{\mathbb{S}^n}\log \hat u_{e(t)}(x,t)d\theta+C(n,\tilde{X}_0)e^{-\frac{2(n+1)}{2n+1}t}\\
\ge& \dashint_{\mathbb{S}^n}\log \hat u_{e(t_1)}(x,t_1)d\theta+C(n,\tilde{X}_0)e^{-\frac{2(n+1)}{2n+1}t_1} \\
= & \mathcal{E}(\hat{\Omega}_{t_1})+C(n,\tilde{X}_0)e^{-\frac{2(n+1)}{2n+1}t_1}.
\end{align*}
The continuity argument can be applied to conclude the same for all $t_0\le t\le t_1$,
which implies that
\begin{equation}
\begin{aligned}
( \mathcal{E}(\hat{\Omega}_{t_1})+C(n,\tilde{X}_0)e^{-\frac{2(n+1)}{2n+1}t_1})-(\mathcal{E}(\hat{\Omega}_{t_2})+C(n,\tilde{X}_0)e^{-\frac{2(n+1)}{2n+1}t_2})
\le 0,
\end{aligned} 
\end{equation}
for any $t_1\ge t_2\ge t_0$.
Then
$\mathcal{E}(\hat{\Omega}_t)+C(n,\tilde{X}_0)e^{-\frac{2(n+1)}{2n+1}t}$ is non-increasing when $t\ge t_0$ for sufficiently large $t_0$,  which proves the first claim for $\alpha=1$.
Since $\mathcal{E}(\hat{\Omega}_{t})$ is bounded below by zero  by Proposition 3.1 in \cite{Guan-Ni}, it follows that $\lim\limits_{t \to \infty}\big(\mathcal{E}(\hat{\Omega}_t)+C(n,\tilde{X}_0)e^{-\frac{2(n+1)}{2n+1}t}\big)$ exists. 	

Case 2: $\alpha \neq1$ and $\alpha \ge \frac{1}{n+2}$.

\begin{align*}
\frac{d}{dt}\Big(\frac{\alpha}{\alpha-1}\dashint_{\mathbb{S}^n}\hat u^{1-\frac{1}{\alpha}}_{e(t)}(x,t)d\theta\Big) &=\frac{\dashint_{\mathbb{S}^n}\hat {u}^{-\frac{1}{\alpha}}_{e(t)}(\hat {u}_{e(t)}-\frac{\psi \hat{K}^{\alpha}}{\dashint_{\mathbb{S}^n}\psi \hat{K}^{\alpha-1}d\theta})d\theta}{\dashint_{\mathbb{S}^n}\hat {u}^{1-\frac{1}{\alpha}}_{e(t)}d\theta}\\
&\le 1-\frac{\dashint_{\mathbb{S}^n}\hat{K}^{\alpha}\hat{u}^{-\frac{1}{\alpha}}_{e(t)}d \theta}{\dashint_{\mathbb{S}^n}\hat u^{1-\frac{1}{\alpha}}_{e(t)}d \theta\dashint_{\mathbb{S}^n}\hat K^{\alpha-1}d \theta}(1-C(n,\alpha,\tilde{X}_0)e^{-\frac{2(n+1)}{2n+1}t})\\
&= \big(1-\frac{\dashint_{\mathbb{S}^n}\hat{K}^{\alpha}\hat{u}^{-\frac{1}{\alpha}}_{e(t)}d \theta}{\dashint_{\mathbb{S}^n}\hat{u}^{1-\frac{1}{\alpha}}_{e(t)}d \theta\dashint_{\mathbb{S}^n}\hat K^{\alpha-1}d \theta}\big)\big(1-C(n,\alpha,\tilde{X}_0)e^{-\frac{2(n+1)}{2n+1}t}\big)\\
&+C(n,\alpha,\tilde{X}_0)e^{-\frac{2(n+1)}{2n+1}t}
\end{align*}
for some positive constant $C(n,\alpha,\tilde{X}_0)$, where we use (\ref{decayest}) in the second step. Choose $t_0$ such that $1-C(n,\alpha,\tilde{X}_0)e^{-\frac{2(n+1)}{2n+1}t_0}\ge \frac{1}{2}$ and $t_0\ge T^*-\delta_1$. This implies that there exists a $\delta>0$ such that for $t\in (t_1-\delta,t_1)$,
\begin{align*}
&\mathcal{E}_{\alpha}(\hat{\Omega}_t)+C(n,\alpha,\tilde{X}_0)e^{-\frac{2(n+1)}{2n+1}t}\\
\ge& \frac{\alpha}{\alpha-1}\dashint_{\mathbb{S}^n}\hat u^{1-\frac{1}{\alpha}}_{e(t)}(x,t)d \theta+C(n,\alpha,\tilde{X}_0)e^{-\frac{2(n+1)}{2n+1}t}\\
\ge& \frac{\alpha}{\alpha-1}\dashint_{\mathbb{S}^n}\hat u^{1-\frac{1}{\alpha}}_{e(t_1)}(x,t_1)d \theta+C(n,\alpha,\tilde{X}_0)e^{-\frac{2(n+1)}{2n+1}t_1}\\
= & \mathcal{E}_{\alpha}(\hat{\Omega}_{t_1})+C(n,\alpha,\tilde{X}_0)e^{-\frac{2(n+1)}{2n+1}t_1}.
\end{align*}
The continuity argument can be applied to conclude the same for all $t_0\le t\le t_1$,
which implies that
\begin{align*}
(\mathcal{E}_{\alpha}(\hat{\Omega}_{t_1})+C(n,\alpha,\tilde{X}_0)e^{-\frac{2(n+1)}{2n+1}t_1})-(\mathcal{E}_{\alpha}(\hat{\Omega}_{t_2})+C(n,\alpha,\tilde{X}_0)e^{-\frac{2(n+1)}{2n+1}t_2})
\le 0,
\end{align*} 
for any $t_1\ge t_2\ge t_0$.
Then
$\mathcal{E}_{\alpha}(\hat{\Omega}_t)+C(n,\alpha,\tilde{X}_0)e^{-\frac{2(n+1)}{2n+1}t}$ is non-increasing when $t\ge t_0$ for sufficiently large $t_0$,  which proves the first claim for $\alpha\ge \frac{1}{n+2}$ and $\alpha\ne 1$.
Since $\mathcal{E}_{\alpha}(\hat{\Omega}_{t})$ is bounded below by zero  by Corollary 2.2 in \cite{AGN}, it follows that $\lim\limits_{t \to \infty}\big(\mathcal{E}_{\alpha}(\hat{\Omega}_t)+C(n,\alpha,\tilde{X}_0)e^{-\frac{2(n+1)}{2n+1}t}\big)$ exists. 
 \end{proof}
\begin{remark}
The method to modify the "standard entropy" by adding a correction term here to obtain a proper monotone entropy follows the argument in \cite{CGH}, where in-homogeneous Gauss curvature type flows was treated.
\end{remark}
\subsection{$C^0$ estimates}
\begin{corollary} \label{upper bound and lower bound}
Let $\hat{X}(t)=\partial \hat{\Omega}_t$ be a solution to (\ref{nor X equa eucl}) with $|\hat{\Omega}_t|=B(1)$ and $\alpha>\frac{1}{n+2}$, there exists $C=C(n,\alpha,\tilde{X}_0)$ such that
\begin{equation}\label{omega bound}
\max\{\omega_+(\hat{\Omega}_t),r_+(\hat{\Omega}_t)\}\le C,\qquad	\min\{\omega_-(\hat{\Omega}_t),r_-(\hat{\Omega}_t)\}\ge \frac{1}{C},	
\end{equation}
for all $t\ge t_0$, where $t_0$ is defined in Theorem \ref{monotone quantity}. Here $\omega_+(\hat{\Omega}_t)$ and $\omega_-(\hat{\Omega}_t)$ are width functions (see e.g. \cite{AGN}).
\end{corollary}
\begin{proof}
By Proposition 2.7 in \cite{AGN}, we have 
\[\min\{\omega_-(\hat{\Omega}_t),r_-(\hat{\Omega}_t)\}\ge C^{-1}e^{-\beta\mathcal{E}_{\alpha}(\hat{\Omega}_t)}\ge  C^{-1}e^{-\beta\big(\mathcal{E}_{\alpha}(\hat{\Omega}_t)+C(n,\alpha,\tilde{X}_0)e^{-\frac{2(n+1)}{2n+1}t} \big)}\]
for some positive constant $\beta$ depending only on $n,\alpha$. Since the entropy $\mathcal{E}_{\alpha}(\hat \Omega_t)+C(n,\alpha,\tilde{X}_0)e^{-\frac{2(n+1)}{2n+1}t}$ is non-increasing for all $t\ge t_0$, this proves the second inequality of (\ref{omega bound}). The first inequality follows from (\ref{vol and r}).

\end{proof}
\begin{lemma}\label{entropy in zero point}
Let $\hat u(x,t)$ be the solution of (\ref{norm u evo general}). 

(1) when $\alpha=1$, we have
\begin{equation}\label{entropy value in zero point a}
\begin{aligned}
&\mathcal{E}(\hat{\Omega}_{t},0)+C(n,\tilde{X}_0)e^{-\frac{2(n+1)}{2n+1}t}-\mathcal{E}^{\infty}\\
\ge & \int^{\infty}_{t}\dashint_{\mathbb{S}^n}\big(\sqrt{\frac{\hat u}{\hat{K}}}-\sqrt{\frac{\hat{K}}{\hat u}}\big)^2\big(1-C(n,\tilde{X}_0)e^{-\frac{2(n+1)}{2n+1}t}\big)d\theta dt
\end{aligned}
\end{equation}
 for all $ t\ge t_0$;

(2) when $\alpha \neq1$ and $\alpha \ge \frac{1}{n+2}$,  we have
\begin{equation}\label{entropy value in zero point b}
\begin{aligned}
&\mathcal{E}_{\alpha}(\hat{\Omega}_{t},0)+C(n,\alpha,\tilde{X}_0)e^{-\frac{2(n+1)}{2n+1}t}- \mathcal{E}^{\infty}_{\alpha}\\
&\ge  \int^{\infty}_{t} \big(1-\frac{\dashint_{\mathbb{S}^n}(\hat u)^{-\frac{1}{\alpha}}\hat{K}^{\alpha}d\theta}{\dashint_{\mathbb{S}^n}(\hat u)^{1-\frac{1}{\alpha}}\hat{K}^{\alpha-1}d\theta\cdot \dashint_{\mathbb{S}^n}(\hat u)^{1-\frac{1}{\alpha}}d\theta}\big)(1-C(n,\alpha,\tilde{X}_0)e^{-\frac{2(n+1)}{2n+1}t})dt
\end{aligned}
\end{equation} 
 for all $ t\ge t_0$. 
 Here $t_0$ is defined in Theorem \ref{monotone quantity}. 

\begin{proof}
 We adopt the argument in \cite{Guan-Ni}.

For each $T_0>t_0$ fixed, pick $T>T_0$. Let $a^T=(a_1^T,\cdots, a^T_{n+1})$ be the entropy point of $\hat{\Omega}_T$. Set $\hat{u}^T=\hat{u}-e^{t-T}\sum^{n+1}_{i=1}a_i^Tx_i$. It can be checked that 
\begin{equation}\label{evol-uT}
\hat{u}_t^T=\hat{u}^T-\frac{\psi \hat{K}^\alpha}{\dashint_{\mathbb{S}^n}\psi \hat{K}^{\alpha-1}d\theta}.
\end{equation} Since both the origin and the entropy point $a^T$ are in $\Int(\hat{\Omega}_{T})$,
\[|a^T|\le 2r_+(\hat{\Omega}_{T})\le C.\]
If $T$ is large enough, $\hat{u}^T(x,0)>0$ for all $x\in\mathbb{S}^n$. We also know that  $\hat{u}^T(x,T)>0$ for all $x\in\mathbb{S}^n$ since the entropy point is an interior point of $\hat\Omega_T$. If $\hat{u}^T(x_1,t_1)\le 0$ for some $0<t_1<T$ and $x_1\in\mathbb{S}^n$, then (\ref{evol-uT})  implies $\hat{u}^T(x_1,t)<0$ for all $t>t_1$, which contradicts $\hat{u}^T(x,T)>0$. Hence $\hat{u}(x,t)>0$ for all $0\le t\le T$ and all $x\in\mathbb{S}^n$. A similar calculation to that in Theorem \ref{monotone quantity} shows 

Case 1: $\alpha =1$.

\begin{align*}
\frac{d}{dt}\dashint_{\mathbb{S}^n}\log \hat{u}^T(x,t)d\theta &=\dashint_{\mathbb{S}^n}\frac{\hat{u}^T-\frac{\psi\hat{K}}{\dashint_{\mathbb{S}^n}\psi d\theta}}{\hat{u}^T}d\theta\\
&\le 1-\dashint_{\mathbb{S}^n}\frac{\hat{K}}{\hat u^T}(1-C(n,\tilde{X}_0)e^{-\frac{2(n+1)}{2n+1}t})d\theta \\
&\leq-\dashint_{\mathbb{S}^n}\big(\sqrt{\frac{\hat{u}^T(x,t)}{\hat{K}(x,t)}}-\sqrt{\frac{\hat{K}(x,t)}{\hat{u}^T(x,t)}}\big)^2d\theta\big(1-C(n,\tilde{X}_0)e^{-\frac{2(n+1)}{2n+1}t}\big) \\
&+C(n,\tilde{X}_0)e^{-\frac{2(n+1)}{2n+1}t}
\end{align*}
for some positive constant $C(n,\tilde{X}_0)$. Hence
\begin{align*}
&(\dashint_{\mathbb{S}^n}\log \hat{u}^T(x,t_0)d\theta+C(n,\tilde{X}_0)e^{-\frac{2(n+1)}{2n+1}t_0})-(\mathcal{E}(\hat{\Omega}_T)+C(n,\tilde{X}_0)e^{-\frac{2(n+1)}{2n+1}T})\\
\ge & \int^{T}_{t_0}\dashint_{\mathbb{S}^n}\big(\sqrt{\frac{\hat{u}^T(x,t)}{\hat{K}(x,t)}}-\sqrt{\frac{\hat{K}(x,t)}{\hat{u}^T(x,t)}}\big)^2\big(1-C(n,\tilde{X}_0)e^{-\frac{2(n+1)}{2n+1}t}\big)d\theta dt.\\
\end{align*}
Since $T_0\le T$,
\begin{align*}
&(\dashint_{\mathbb{S}^n}\log \hat{u}^T(x,t_0)d\theta+C(n,\tilde{X}_0)e^{-\frac{2(n+1)}{2n+1}t_0})-(\mathcal{E}(\hat{\Omega}_T)+C(n,\tilde{X}_0)e^{-\frac{2(n+1)}{2n+1}T})\\
\ge & \int^{T_0}_{t_0}\dashint_{\mathbb{S}^n}\big(\sqrt{\frac{\hat{u}^T(x,t)}{\hat{K}(x,t)}}-\sqrt{\frac{\hat{K}(x,t)}{\hat{u}^T(x,t)}}\big)^2\big(1-C(n,\tilde{X}_0)e^{-\frac{2(n+1)}{2n+1}t}\big)d\theta dt.
\end{align*}
 Now let $T\rightarrow \infty$; as $\hat{u}^T(x,t)\rightarrow \hat{u}(x,t)$ uniformly for $t_0\le t\le T_0, x\in \mathbb{S}^n$, we obtain
\begin{align*}
&(\dashint_{\mathbb{S}^n}\log \hat{u}(x,t_0)d\theta+C(n,\tilde{X}_0)e^{-\frac{2(n+1)}{2n+1}t_0})-\mathcal{E}^{\infty}\\
\ge & \int^{T_0}_{t_0}\dashint_{\mathbb{S}^n}\big(\sqrt{\frac{\hat{u}(x,t)}{\hat{K}(x,t)}}-\sqrt{\frac{\hat{K}(x,t)}{\hat{u}(x,t)}}\big)^2\big(1-C(n,\tilde{X}_0)e^{-\frac{2(n+1)}{2n+1}t}\big)d\theta dt.
\end{align*}
That is 
\begin{align*}
&(\dashint_{\mathbb{S}^n}\log \hat{u}(x,t_0)d\theta+C(n,\tilde{X}_0)e^{-\frac{2(n+1)}{2n+1}t_0})-\mathcal{E}^{\infty}\\
\ge & \int^{T_0}_{t_0}\dashint_{\mathbb{S}^n}\big(\sqrt{\frac{\hat{u}(x,t)}{\hat{K}(x,t)}}-\sqrt{\frac{\hat{K}(x,t)}{\hat{u}(x,t)}}\big)^2\big(1-C(n,\tilde{X}_0)e^{-\frac{2(n+1)}{2n+1}t}\big)d\theta dt.
\end{align*}

Case 2: $\alpha \neq1$ and $\alpha \ge \frac{1}{n+2}$.

By (\ref{evol-uT}), we have
\begin{align*}
&\frac{d}{dt} \Big(\frac{\alpha}{\alpha-1}\big(\log{\dashint_{\mathbb{S}^n}(\hat{u}^T)^{1-\frac{1}{\alpha}}}d\theta\big)\Big)\\
=&1-\frac{\dashint_{\mathbb{S}^n}(\hat{u}^T)^{-\frac{1}{\alpha}}\psi\hat{K}^{\alpha}d\theta}{\dashint_{\mathbb{S}^n}\psi\hat{K}^{\alpha-1}d\theta\cdot \dashint_{\mathbb{S}^n}(\hat{u}^T)^{1-\frac{1}{\alpha}}d\theta}\\
\le& 1-\frac{\dashint_{\mathbb{S}^n}(\hat{u}^T)^{-\frac{1}{\alpha}}\hat{K}^{\alpha}d\theta}{\dashint_{\mathbb{S}^n}\hat{K}^{\alpha-1}d\theta\cdot \dashint_{\mathbb{S}^n}(\hat{u}^T)^{1-\frac{1}{\alpha}}d\theta}(1-C(n,\alpha,\tilde{X}_0)e^{-\frac{2(n+1)}{2n+1}t})\\
\le& \big(1-\frac{\dashint_{\mathbb{S}^n}(\hat{u}^T)^{-\frac{1}{\alpha}}\hat{K}^{\alpha}d\theta}{\dashint_{\mathbb{S}^n}\hat{K}^{\alpha-1}d\theta\cdot \dashint_{\mathbb{S}^n}(\hat{u}^T)^{1-\frac{1}{\alpha}}d\theta}\big)(1-C(n,\alpha,\tilde{X}_0)e^{-\frac{2(n+1)}{2n+1}t})\\
&+C(n,\alpha,\tilde{X}_0)e^{-\frac{2(n+1)}{2n+1}t}\\
\end{align*}
for some positive constant $C(n,\alpha,\tilde{X}_0)$. Hence
\begin{align*}
&\frac{\alpha}{\alpha-1}\log{\dashint_{\mathbb{S}^n}(\hat{u}^T)^{1-\frac{1}{\alpha}}}(x,t_0)d\theta+C(n,\alpha,\tilde{X}_0)e^{-\frac{2(n+1)}{2n+1}t_0}\\
&-\Big(\mathcal{E}_{\alpha}(\hat{\Omega}_T)+C(n,\alpha,\tilde{X}_0)e^{-\frac{2(n+1)}{2n+1}T}\Big)\\
\ge& \int^T_{t_0} \big(\frac{\dashint_{\mathbb{S}^n}(\hat{u}^T)^{-\frac{1}{\alpha}}\hat{K}^{\alpha}d\theta}{\dashint_{\mathbb{S}^n}\hat{K}^{\alpha-1}d\theta\cdot \dashint_{\mathbb{S}^n}(\hat{u}^T)^{1-\frac{1}{\alpha}}d\theta}-1\big)(1-C(n,\alpha,\tilde{X}_0)e^{-\frac{2(n+1)}{2n+1}t})dt\\
\ge& \int^{T_0}_{t_0} \big(\frac{\dashint_{\mathbb{S}^n}(\hat{u}^T)^{-\frac{1}{\alpha}}\hat{K}^{\alpha}d\theta}{\dashint_{\mathbb{S}^n}\hat{K}^{\alpha-1}d\theta\cdot \dashint_{\mathbb{S}^n}(\hat{u}^T)^{1-\frac{1}{\alpha}}d\theta}-1\big)(1-C(n,\alpha,\tilde{X}_0)e^{-\frac{2(n+1)}{2n+1}t})dt.
\end{align*}
Let $T\rightarrow \infty$, $\hat{u}^T(x,t)\rightarrow \hat{u}(x,t)$ uniformly for $t_0\le t\le T_0, x\in \mathbb{S}^n$. Similarly, we obtain
\begin{equation}
\begin{aligned}
&\mathcal{E}_{\alpha}(\hat{\Omega}_{t_0},0)+C(n,\alpha,\tilde{X}_0)e^{-\frac{2(n+1)}{2n+1}t_0}- \mathcal{E}^{\infty}_{\alpha}\\
\ge&  \int^{T_0}_{t_0} \big(\frac{\dashint_{\mathbb{S}^n}(\hat{u})^{-\frac{1}{\alpha}}\hat{K}^{\alpha}d\theta}{\dashint_{\mathbb{S}^n}\hat{K}^{\alpha-1}d\theta\cdot \dashint_{\mathbb{S}^n}(\hat{u})^{1-\frac{1}{\alpha}}d\theta}-1\big)(1-C(n,\alpha,\tilde{X}_0)e^{-\frac{2(n+1)}{2n+1}t})dt.
\end{aligned}
\end{equation}
Now, (\ref{entropy value in zero point a}) and (\ref{entropy value in zero point b}) could be established for $t=t_0$, since $T_0$ is arbitrary.
If in the above, we replace $t_0$ by any $t_0\le t\le T$, we obtain (\ref{entropy value in zero point a}) and (\ref{entropy value in zero point b}). 
\end{proof}
\end{lemma}
Define the following collection of convex bodies:
\begin{equation}
\Gamma_{\rho}=\Big{\{}\Omega\subset \mathbb{R}^{n+1} \text{compact, convex} \Big{|}\{r_+(\Omega),r_-(\Omega)\subset [\rho,\frac{1}{\rho}]\}\Big\},
\end{equation}
where $\rho\in (0,1)$.

\begin{theorem} \label{lower bound of u}
	Suppose $\alpha>\frac{1}{n+2}$ and $\hat{u}(x,t)>0$ is the solution of (\ref{norm u evo general}) with initial data $\hat{u}_0$, where $\hat{u}_0$ is the support function of the convex body $\hat{\Omega}_0$ (enclosed by $\hat{X}_0$) with $|\hat{\Omega}_0|=|B(1)|$. Then there exists $\epsilon=\epsilon(n,\mathcal{E}(\hat{\Omega}_0))>0$  and $T_0$ (depends on $n,\alpha, \tilde{X}_0, T^*-\delta_1$) such that for $t\ge T_0$ and $x\in \mathbb{S}^n,$
	\begin{equation}
		\hat{u}(x,t)\ge \epsilon.
	\end{equation}
\end{theorem}
\begin{proof}
Note that under (\ref{norm u evo general}) for $\alpha>\frac{1}{n+2}$, we have $|\hat{\Omega}_t|=|B(1)|$ for $t\geq \Big(\frac{B(1)}{|\hat{\tilde{\Omega}}(T^*-\delta_1)|}\Big)^{\frac{1}{n+1}}$. By Corollary \ref{upper bound and lower bound}, there exists $\rho>0$ such that $\hat{\Omega}_t\in \Gamma_{\rho}$ for every $t\ge t_0$.
By Lemma \ref{entropy in zero point}, we have 
\[\mathcal{E}^{\infty}_{\alpha}-(\mathcal{E}_{\alpha}(\hat{\Omega}_{t})+C(n,\alpha,\tilde{X}_0)e^{-\frac{2(n+1)}{2n+1}t})\le \mathcal{E}^{\infty}_{\alpha}-(\mathcal{E}_{\alpha}(\hat{\Omega}_{t},0)+C(n,\alpha,\tilde{X}_0)e^{-\frac{2(n+1)}{2n+1}t})\le 0.\]
This implies that $\lim\limits_{t \to \infty}\mathcal{E}_{\alpha}(\hat{\Omega}_{t},0)=\mathcal{E}^{\infty}_{\alpha}$ and $\lim\limits_{t \to \infty}(\mathcal{E}_{\alpha}(\hat{\Omega}_{t},0)-\mathcal{E}_{\alpha}(\hat{\Omega}_{t}))=0$. Let $z_e(\hat{\Omega}(t))$ be the entropy point of $\hat{\Omega}(t)$. By Lemma 4.2 in \cite{AGN}, there exists a constant $D>0$ such that 
\[|z_e(\hat{\Omega}(t))-0|^2\le \frac{1}{D}|\mathcal{E}_{\alpha}(\hat{\Omega}_{t},0)-\mathcal{E}_{\alpha}(\hat{\Omega}_{t})|,\]
when $t\ge T_0$ for $T_0$ sufficiently large. The right hand side approaches zero as $t\rightarrow \infty$. The claimed result then follows from Lemma 4.4 in \cite{AGN}.
\end{proof}
\begin{proposition}\label{upper bound and lower bound u}
For the normalized flow (\ref{norm u evo general}) with $\alpha>\frac{1}{n+2}$, there exists $\Lambda=\Lambda(n,\alpha,\tilde{X}_0)>0$ such that 
	\begin{equation}
	\frac{1}{\Lambda}\le \hat{u}(x,t)\leq \Lambda
	\end{equation}
for $t\ge T_0$.	
\end{proposition}
\begin{proof}
The upper bound is immediate since the diameter of $\hat{\Omega}_t$ is bounded by Corollary \ref{upper bound and lower bound}. The lower bound is provided by Theorem \ref{lower bound of u}.
\end{proof}
\subsection{$C^2$-estimates}
\begin{theorem}\label{upper bound of norm K}
Suppose $\alpha>\frac{1}{n+2}$ and $\hat{u}(x,t)>0$ is the solution of (\ref{norm u evo general}) with initial data $\hat{u}_0$, where $\hat{u}_0$ is the support function of the convex body $\hat{\Omega}_0$ with $|\hat{\Omega}_0|=|B(1)|$. Then there exists a constant $\bar{K}=\bar{K}(n,\alpha, \tilde{X}_0)>0$ such that 
	\begin{equation}
	\hat{K}(x,t)\leq \bar{K}	
	\end{equation}
	for $t\geq t_0$.
\end{theorem}
\begin{proof}
	This is immediate by re-scaling the upper bound of $\hat{\tilde{K}}$ obtained in Lemma \ref{unno K above} and Corollary \ref{upper bound and lower bound}. 
\end{proof}
\begin{theorem}\label{lower bound of normal K}
Suppose $\alpha>\frac{1}{n+2}$ and $\hat{u}(x,t)>0$ is the solution of (\ref{norm u evo general}) with initial data $\hat{u}_0$, where $\hat{u}_0$ is the support function of the convex body $\hat{\Omega}_0$ with $|\hat{\Omega}_0|=|B(1)|$. Then there exists a constant $\underline{K}=\underline{K}(n,\alpha, \tilde{X}_0)>0$ such that   
	\begin{equation}
	\hat{K}(x,t)\geq \underline{K}	
	\end{equation}
	 for $t\ge t_1(n,\alpha,\tilde{X_0})\ge T_0$.
\end{theorem}
\begin{proof}
It follows the same lines of the alternate proof of Theorem 5.2 in \cite{AGN}.
	Let 
	\begin{align*}
		f=\psi \hat{K}^\alpha=(\hat{u}-\hat{u}_t)\eta,
	\end{align*} 
	where $\eta=\dashint_{\mathbb{S}^n}\psi \hat{K}^{\alpha-1} d\theta$. Let $\hat{W}=(\hat{u}_{ij}+\hat{u}{\delta_{ij}})$, and $\mathcal{L}:=\partial_t-\frac{\alpha \psi \sigma_n^{-\alpha-1}}{\eta}\sigma_{n}^{ij}\nabla_i\nabla_j$. Then
	\begin{equation}
	\begin{aligned}
		f_t=&(\psi \hat{K}^\alpha)_t\\
		=&\psi_t\hat{K}^\alpha-\alpha \psi\sigma_n^{-\alpha-1}\sigma_{n}^{ij}(\hat{u}_{ij}-\frac{f_{ij}}{\eta}+\hat{u}\delta_{ij}-\frac{f}{\eta}\delta_{ij})\\
		=&\psi_t\hat{K}^\alpha-n\alpha f+\frac{\alpha \psi \sigma_n^{-\alpha-1}}{\eta}\sigma_{n}^{ij}f_{ij}+\frac{\alpha f\sigma_n^{-1}\sigma_{n-1}}{\eta}f.   	
	\end{aligned}
	\end{equation}
	Since 
	\begin{equation}
	\begin{aligned}
		\mathcal{L}\hat{u}=&\hat u-\frac{f}{\eta}-\frac{\alpha \psi\sigma_n^{-\alpha-1}}{\eta}\sigma_n^{ij}(\hat u_{ij}+\hat u\delta_{ij})+\hat u\frac{\alpha \psi\sigma_n^{-\alpha-1}}{\eta}\sigma_n^{ij}\delta_{ij}\\
		=&\hat u-\frac{(n\alpha+1)f}{\eta}+\frac{\alpha f\sigma_n^{-1}\sigma_{n-1}}{\eta}\hat u,
		\end{aligned}
	\end{equation}
	thus
	\begin{equation}
		\begin{aligned}
			\mathcal{L}(\log(f\hat{u}^l))\geq &\frac{\mathcal{L}f}{f}+l\frac{\mathcal{L}\hat u}{\hat u}	\\	
			= &(\frac{\psi_t}{\psi}-n\alpha+l)-\frac{l(n\alpha+1)f}{\hat u\eta}+(l+1)\frac{\alpha f\sigma_n^{-1}\sigma_{n-1}}{\eta}.
		\end{aligned}
	\end{equation}
	Since at the minimum point of $\log(fu^l)$, we have 
\begin{equation}
\frac{f_i}{f}+l\frac{\hat u_i}{\hat u}=0.
\end{equation}
Thus
\begin{equation}
\hat u_{ti}=(1+\frac{lf}{\hat u\eta})\hat u_i,	
\end{equation}	
and 
\begin{equation}
		\begin{aligned}
			\psi_t=&\frac{1}{e^{t}}[\psi_{\hat{\tilde{u}}}(\hat u_t-\hat u)+\psi_{\hat {\tilde{u}}_i}(\hat u_{it}-\hat u_i)]\\
			=&\frac{1}{e^{t}}[\psi_{\hat {\tilde{u}}}(\hat u-\frac{f}{\eta}-\hat u)+\psi_{\hat{\tilde{u}}_i}((1+\frac{lf}{\hat{u}\eta})\hat u_i-\hat u_i)]\\
			=&\frac{1}{e^t}[-\hat u\psi_{\hat{\tilde{u}}}+l\psi_{\hat{\tilde{u}}_i}\hat{u}_i]\frac{f}{\hat{u}\eta}.\\
		\end{aligned}
\end{equation}
Let $0<\varepsilon<1$ small, note that
\begin{equation}
	\begin{aligned}
		\psi_{\hat{\tilde{u}}}=&2\kappa\psi[(\frac{n+2}{2}\alpha+\frac{1}{2})(1+\kappa\frac{\hat{r}^2}{e^{2t}})^{-1}+(-\frac{n+2}{2}\alpha+\frac{1}{2})(1+\kappa\frac{\hat u^2}{e^{2t}})^{-1}]\frac{\hat u}{e^t}\leq\varepsilon \psi,\\
		|\psi_{\hat{\tilde{u}}_i}|=&|2\kappa\psi (\frac{n+2}{2}\alpha+\frac{1}{2})(1+\kappa\frac{\hat r^2}{e^{2t}})^{-1}\frac{\hat u_i}{e^t}|\leq\varepsilon\psi
	\end{aligned}
\end{equation}
for $t>T_1(n,\alpha,\tilde{X}_0)\geq T_0$ large enough since $\frac{\hat u}{e^t}=\hat{\tilde{u}},\frac{|\hat u_i|}{e^t}\leq \frac{\hat r}{e^t}=\hat{\tilde{r}}\to 0$ as $t\to\infty$. Thus
\begin{equation}
	\psi_t\geq-\frac{\psi f}{\hat u\eta} 
\end{equation}
for $t>T_2(n,\alpha,\tilde{X}_0)\geq T_1$ large enough, which implies that
\begin{equation}
	\begin{aligned}
		\mathcal{L}(\log(f\hat u^l))\geq&(l-n\alpha)-\frac{l(n\alpha+2)f}{\hat u\eta}.
	\end{aligned}
\end{equation}
Since $\eta=\dashint_{\mathbb{S}^n}\psi \hat{K}^{\alpha-1}d\theta\geq\frac{1}{2}\dashint_{\mathbb{S}^n}\hat{K}^{\alpha-1}d\theta$ for $t$ large and $(\dashint_{\mathbb{S}^n}\hat{K}^{\alpha-1})^{\frac{1}{\alpha-1}}\geq e^{\mathcal{E}_\alpha(\hat{\Omega}_t)}$ by Lemma 5.4 of \cite{AGN}. Thus $\eta\geq \frac{1}{2}e^{(\alpha-1)\mathcal{E}_\alpha(\hat{\Omega}_t)}\geq\frac{1}{2}$ for $\alpha\geq 1$ and $\eta\geq \frac{1}{2}\Lambda^{\alpha-1}$ for $\alpha<1$.	
	Let $\lambda=\frac{100}{(n\alpha+2)^2\Lambda^{(n-1)\alpha+4}}$, take $\Lambda$ large enough such that $\min_{t=0}(f\hat u^l)\geq 2\lambda$. Then we claim that $\min(f\hat{u}^l)\geq \lambda >0$ for all $t>0$. In fact, suppose $t'$ is the first time when $\min(f\hat u^l)=f\hat u^l(x',t')$ touch $\lambda$, take $l=n\alpha+2$, then at $(x',t')$,  we have
	\begin{equation}
		0\geq 2-(n\alpha+2)^2\frac{f}{\hat u\eta}=2-(n\alpha+2)^2\frac{f\hat u^l}{\eta \hat u^{l+1}}\geq 2-2(n\alpha+2)^2\frac{f\hat u^l}{\Lambda^{\alpha-l-2}},
	\end{equation}
where we used the fact that $\hat u\geq\frac{1}{\Lambda}$ in the last inequality. Thus $f\hat u^l(x',t')\leq \frac{\Lambda^{\alpha-2-l}}{(n\alpha+2)^2}=\frac{1}{(n\alpha+2)^2\Lambda^{(n-1)\alpha+4}}=\frac{1}{100}\lambda$ which is a contradiction since $f\hat u^l(x',t')=\lambda$. Thus the claim is true and $\hat{K}\geq(\frac{\lambda}{\hat u^l})^{\frac{1}{\alpha}}\geq \underline{K}>0$.
\end{proof} 

\begin{theorem}\label{upper and lower bound of normal W thm}
Suppose $\alpha>\frac{1}{n+2}$ and $\hat{u}(x,t)>0$ is the solution of (\ref{norm u evo general}) with initial data $\hat{u}_0$, where $\hat{u}_0$ is the support function of the convex body $\hat{\Omega}_0$ with $|\hat{\Omega}_0|=|B(1)|$. Then there exist constants $C_1,C_2$ depending only on $n,\alpha, \tilde{X}_0>0$ such that
\begin{equation}\label{upper and lower bound of normal W eq}
	C_1I\leq \hat{W}_{ij}=\hat{u}_{ij}+\hat{u}\delta_{ij}\leq C_2I.
\end{equation}
for $t\geq t_1>0$.  
\end{theorem}
\begin{proof}
Since we already have the upper and lower bound of $\hat{K}$, it suffices to prove an upper bound of the eigenvalues of $(\hat{W}_{ij})$. Similar to the proof of Lemma \ref{unor ki lower boud}, let $(\lambda_1,\dots,\lambda_n)$ be the eigenvalues of $(\hat{W}_{ij})$, $\lambda(x,t):=\max_{i=1,\dots,n}\lambda_i(x,t)$.  For any $T>0$, suppose $\lambda$ attains its maximum on $\mathbb{S}^n\times [0,T]$ at $(x',t')$. We take a local orthonormal frame $\{e_1,\dots,e_n\}$ on $\mathbb{S}^n$ around $x'$, such that $(\hat{W}_{ij}(x',t'))$ is diagonal and $\lambda(x',t')=\lambda_1(x',t')=\hat{W}_{11}(x',t')$. Then $\hat{W}_{11}(x,t)$ also attains its maximum at $(x',t')$. If $t'=0$, $\lambda(x,t)\leq\lambda (x',0)$, we are done. Assume that $t'>0$,  by (\ref{unor W11 evo}), we have  
	\begin{align*}
		\hat{W}_{11,t}(x,t)=&(e^t\hat{\tilde{W}}_{11}(x,\tau))_t=\hat{W}_{11}(x,t)+e^t\hat{\tilde{W}}_{11,\tau}(x,\tau)\tau'(t)\\
		=&\hat{W}_{11}(x,t)+\frac{e^{-n\alpha t}}{\eta}\hat{\tilde{W}}_{11,\tau}(x,\tau)\\
		=&\hat{W}_{11}-\frac{1}{\eta}\hat{K}^\alpha[\psi+\psi_{\hat{\tilde{u}}}\hat{\tilde{W}}_{11}-\psi_{\hat{\tilde{u}}}\hat{\tilde{u}}+\psi_{\hat{\tilde{u}}_i}\hat{\tilde{W}}_{11i}-\psi_{\hat{\tilde{u}}_1}\hat{\tilde{u}}_1\\
		&+\psi_{\hat{\tilde{u}}\hat{\tilde{u}}}\hat{\tilde{u}}_1^2+\psi_{\hat{\tilde{u}}_1\hat{\tilde{u}}_1}\hat{\tilde{W}}_{11}^2-2\psi_{\hat{\tilde{u}}_1\hat{\tilde{u}}_1}\hat{\tilde{W}}_{11}\hat{\tilde{u}}+\psi_{\hat{\tilde{u}}_1\hat{\tilde{u}}_1}\hat{\tilde{u}}^2+2\psi_{\hat{\tilde{u}}\hat{\tilde{u}}_1}\hat{\tilde{W}}_{11}\hat{\tilde{u}}_1
			\\
			&-2\psi_{\hat{\tilde{u}}\hat{\tilde{u}}_1}\hat{\tilde{u}}\hat{\tilde{u}}_1-2\alpha \hat{\tilde{W}}^{ii}\hat{\tilde{W}}_{ii1}(\psi_{\hat{\tilde{u}}}\hat{\tilde{u}}_1+\psi_{\hat{\tilde{u}}_1}\hat{\tilde{W}}_{11}-\psi_{\hat{\tilde{u}}_1}\hat{\tilde{u}})+\alpha^2\psi(\hat{\tilde{W}}^{ii}\hat{\tilde{W}}_{ii1})^2\\
			&+\alpha \psi \hat{\tilde{W}}^{ii}\hat{\tilde{W}}^{jj}\hat{\tilde{W}}_{ij1}^2+\alpha \psi \hat{W}^{ii}(\hat{W}_{11}-\hat{W}_{ii})-\alpha \psi \hat{W}^{ii}\hat{W}_{11,ii}],
\end{align*}
since $\hat{\tilde{W}}^{ii}(\hat{\tilde{W}}_{11}-\hat{\tilde{W}}_{ii})$ and $\hat{\tilde{W}}^{ii}\hat{\tilde{W}}_{11,ii}$ are scaling invariant.
Since we already proved that $\hat{\tilde{W}}_{11}\leq C(n,\alpha,\tilde{X}_0)$ in Lemma \ref{unor ki lower boud}, and by maximum principle at $(x',t')$
\begin{equation}
	\hat{\tilde{W}}_{11i}=0,
\end{equation}
thus,
 \begin{align*}
		\mathcal{L}\hat{W}_{11}
		&\leq \hat{W}_{11}-\frac{1}{\eta}\hat{K}^\alpha[-2\alpha \hat{\tilde{W}}^{ii}\hat{\tilde{W}}_{ii1}(\psi_{\hat{\tilde{u}}}\hat{\tilde{u}}_1+\psi_{\hat{\tilde{u}}_1}\hat{\tilde{W}}_{11}-\psi_{\hat{\tilde{u}}_1}\hat{\tilde{u}})+\alpha^2\psi(\hat{\tilde{W}}^{ii}\hat{\tilde{W}}_{ii1})^2\\
			&+\alpha \psi \hat{\tilde{W}}^{ii}\hat{\tilde{W}}^{jj}\hat{\tilde{W}}_{ij1}^2+\alpha \psi \hat{W}^{ii}(\hat{W}_{11}-\hat{W}_{ii})-C(n,\alpha,\tilde{X}_0)]\\
		&\leq \hat W_{11}-\frac{1}{\eta}\hat{K}^\alpha[-C(n,\alpha,\tilde{X}_0)|\hat{\tilde{W}}^{ii}\hat{\tilde{W}}_{ii1}|+\alpha^2\psi(\hat{\tilde{W}}^{ii}\hat{\tilde{W}}_{ii1})^2+\alpha \psi \hat W_{11}\hat W^{ii}\\
		&-C(n,\alpha,\tilde{X}_0)]\\
		&\leq \hat W_{11}-\frac{1}{\eta}\hat{K}^\alpha[\alpha \psi \hat W_{11}\sum_{i}\hat W^{ii}-C(n,\alpha,\tilde{X}_0)]\\
		&\leq \hat W_{11}-C(n,\alpha,\tilde{X}_0,\bar{K},\underline{K})\hat W_{11}\sum_{i}\hat W^{ii}+C(n,\alpha,\tilde{X}_0,\bar{K},\underline{K})\\
		&= \hat W_{11}-C(n,\alpha,\tilde{X}_0,\bar{K},\underline{K})\hat W_{11}\frac{\sigma_{n-1}}{\sigma_{n}}+C(n,\alpha,\tilde{X}_0,\bar{K},\underline{K}),	
		\end{align*}
where we used the Cauchy-Schwartz inequality in the third step, and $ \eta=\dashint_{\mathbb{S}^n}\psi\hat{K}^{\alpha-1}d\theta\leq C(n,\alpha,\tilde{X}_0,\bar{K},\underline{K})$. By Newton 's inequality 
\begin{equation}
	\frac{\sigma_{n-1}}{n}\geq (\frac{\sigma_1}{n})^{\frac{1}{n-1}}\sigma_n^{\frac{n-2}{n-1}},
\end{equation}
we get
\begin{equation}
\begin{aligned}
	 \mathcal{L}\hat{W}_{11}\leq \hat{W}_{11}-C(n,\alpha,\tilde{X}_0,\bar{K},\underline{K})\hat{W}_{11}^{\frac{n}{n-1}}+C(n,\alpha,\tilde{X}_0,\bar{K},\underline{K}).
	\end{aligned}
\end{equation}
This implies that 
\begin{equation}
	\hat W_{11}(x_0,t_0)\leq C(n,\alpha,\tilde{X}_0),
\end{equation}
since $\bar{K}$ and $\underline{K}$ only depend on $n,\alpha,\tilde{X}_0$.
\end{proof} 
Combining Proposition \ref{upper bound and lower bound u}, Theorem \ref{upper and lower bound of normal W thm},  we conclude that there exists a positive constant $C$ depending only on $n, \alpha,\tilde{X}_0$ such that for the unique solution to (\ref{norm u evo general})
\begin{equation}\label{C2 bound}
\|\hat{u}(\cdot,t)\|_{C^2}\leq C.	
\end{equation}

\section{Convergence to a sphere}
Since (\ref{norm u evo general}) is a concave parabolic equation, by Krylov's theorem and the standard theory of parabolic equations, the estimates (\ref{C2 bound}) and (\ref{upper and lower bound of normal W eq}) imply bounds on all derivatives of $\hat{u}(x,t)$. More precisely, for any $k\geq 3$, there exists $C_k\geq 0$, depending only on $n,\alpha,\tilde{X}_0$ such that for $t\geq t_1(n,\alpha,\tilde{X}_0)$, 
\begin{equation}\label{uni ck est}
	\|\hat u(\cdot,t)\|_{C^k(\mathbb{S}^n)}\leq C_k.
\end{equation}
\begin{proposition}
	Let $\hat{X}(t)$ be the solution of (\ref{nor X equa eucl}) with $\alpha>\frac{1}{n+2}$, then $\hat X(x,t)$ converges in $C^\infty$-topology to a round sphere as $t\to\infty$.
\end{proposition}
\begin{proof}
	First, given a sequence $t_j\to\infty$  and $T>0$, define $\hat{u}_j(x,t)=\hat{u}(x,t+t_j)$. Since by (\ref{uni ck est}), $\hat u_j$ are uniformly bounded in $C^k(\mathbb{S}^n\times[0,T])$, for every $k\in\mathbb{N}$. By Arzel\`a-Ascoli theorem, $\hat u_j$ has a subsequence converging in $C^\infty$-topology to a limit $\hat u_\infty$ on $\mathbb{S}^n\times[0,T]$ 
	and $ \hat u_\infty$ is a solution of
	\begin{equation}\label{soli evo}
		\hat u(x,t)_{\infty,t}=\hat u_\infty(x,t)-\frac{\hat{K}^\alpha_{\infty}(x,t)}{\dashint_{\mathbb{S}^n}\hat u_\infty\hat{K}^{\alpha-1}_{\infty}d\theta}
	\end{equation}
on $\mathbb{S}^n\times[0,T]$.

We claim that $\hat u_\infty(x,t)$ is a soliton, i.e.	
\begin{equation}\label{soli eq}
		\lambda(t) \hat u_{\infty}(x,t)=\hat{K}_{\infty}^\alpha(x,t),
	\end{equation}
	 for some $\lambda(t)>0$. In fact, otherwise, there is $(x',t')\in\mathbb{S}^n\times[0,T]$, a sequence $j_k\to\infty$ and positive constants 
	 $\varepsilon, \delta>0$ independent of $j_k$ such that
	 \begin{equation}\label{differ}
	 \begin{aligned}
	 	&\Big(\sqrt{\frac{\hat u(x,t)}{\hat{K}(x,t)}}-\sqrt{\frac{\hat{K}(x,t)}{\hat u(x,t)}}\Big)^2>\delta,\quad(\alpha=1);\\
	 	&\text{or}\\
	 	&1-\frac{\dashint_{\mathbb{S}^n}(\hat u)^{-\frac{1}{\alpha}}\hat{K}^{\alpha}d\theta}{\dashint_{\mathbb{S}^n}(\hat u)^{1-\frac{1}{\alpha}}\hat{K}^{\alpha-1}d\theta\cdot \dashint_{\mathbb{S}^n}(\hat u)^{1-\frac{1}{\alpha}}d\theta}>\delta,\quad (\alpha\neq 1)	
	\end{aligned}
	\end{equation}
	 	 on $B(x',\varepsilon)\times [t'+t_{j_k}-\varepsilon,t'+t_{j_k}+\varepsilon]$ by (\ref{uni ck est}). This implies that 
	 	 	\begin{align*}
	 	 		&\int^{\infty}_{t_0}\dashint_{\mathbb{S}^n}\Big(\sqrt{\frac{\hat u(x,t)}{\hat{K}(x,t)}}-\sqrt{\frac{\hat{K}(x,t)}{\hat u(x,t)}}\Big)^2\Big(1-C(n,\tilde{X}_0)e^{-\frac{2(n+1)}{2n+1}t}\Big)d\theta dt=\infty,\quad (\alpha=1);\\
	 	 		&\text{or}\\
	 	 		&\int^{\infty}_{t_0} \big(1-\frac{\dashint_{\mathbb{S}^n}(\hat u)^{-\frac{1}{\alpha}}\hat{K}^{\alpha}d\theta}{\dashint_{\mathbb{S}^n}(\hat u)^{1-\frac{1}{\alpha}}\hat{K}^{\alpha-1}d\theta\cdot \dashint_{\mathbb{S}^n}(\hat u)^{1-\frac{1}{\alpha}}d\theta}\big)(1-C(n,\alpha,\tilde{X}_0)e^{-\frac{2(n+1)}{2n+1}t})dt=\infty,\quad(\alpha\neq 1),
	 	 	\end{align*}
	  which is a contradiction to (\ref{entropy value in zero point a}) and (\ref{entropy value in zero point b}). Thus (\ref{soli eq}) is true. Multiplying $\hat{K}^{-1}(x,t)$ on both side of (\ref{soli eq}) and integrating on $\mathbb{S}^n$, we get $\lambda(t)=\dashint_{\mathbb{S}^n}\hat{K}^{\alpha-1}_{\infty}d\theta$. Plugging this into (\ref{soli evo}), we get $\hat u_{\infty,t}=0$, $\lambda(t)\equiv\dashint_{\mathbb{S}^n}\hat{K}^{\alpha-1}_{\infty}d\theta$ is a constant. By Theorem 1 of \cite{BCD}, the solution of (\ref{soli eq}) is a round sphere.
	  
Next, we claim that $\{\hat u(x,t)\}$ converges in $C^\infty$-topology to $\hat u_{\infty}$ itself as $t\to\infty$. In fact, otherwise, there is $k\in\mathbb{N}$, a positive constant $\gamma>0$ and a sequence $\{t_l\}\to\infty$ such that 
\begin{equation}\label{sub not con}
\sup_{x\in\mathbb{S}^n}|\hat u^{(k)}(x,t_l)-\hat u_\infty^{(k)}(x,t_l)|\geq \gamma,	\quad \forall\, l\geq 1.
\end{equation}	   
On the other hand, applying the above argument to $\{\hat u_l(x,t):=\hat u(x,t_l+t)\}$, we can find a subsequence $\{\hat u(x,t_{l_j}+t)\}$ of $\{\hat u(x,t_l+t)\}$ converging to $\hat u_\infty(x,t)$ in $C^\infty$-topology on $\mathbb{S}^n\times\{0\}$ as $j\to\infty$, i.e. $\hat u_{t_j}(x,0)$ converges in $C^\infty$-topology to $\hat u_\infty(x)$ on $\mathbb{S}^n$, which is a contradiction to (\ref{sub not con}).  
\end{proof}
Recall that $q_0$ is the extinct point of $\tilde{M}_{\tau}$ as $\tau\to T^*$.
\begin{theorem}\label{norm cover space}
	$\tilde{M}_{\tau}$ converges to a geodesic sphere centered at $q_0$ in $\mathbb{N}^{n+1}(\kappa)$ as $\tau\to T^*$ after the rescaling.  
\end{theorem}
By Corollary \ref{coro to point} and Theorem \ref{norm cover space}, we finish the proof of Theorem \ref{main theo}.

\renewcommand{\abstractname}{Acknowledgements}
\begin{abstract}
The authors would like to thank Professor Pengfei Guan for his supervision and all his useful suggestions.
\end{abstract}


\begin{thebibliography}{99}
\bibitem{A invention}B. Andrews, {\em Gauss curvature flow: the fate of the rolling stones}, Invent. Math. {\bf 138} (1999), no. 1, 151-161.
\bibitem{A pinch} B. Andrews, {\em Motion of hypersurfaces by Gauss curvature}, Pacific J. Math. {\bf 195} (2000), no. 1, 1-34.
\bibitem{curves} B. Andrews, {\em Evolving convex curves}, Calc. Var. Partial Differ. Equa. {\bf 7} (1998) no. 4, 315-371.
\bibitem{affine} B. Andrews, {\em Contraction of convex hypersurfaces by their affine normal}, J. Diff. Geom. {\bf 43} (1996) no.2 , 207-230.
\bibitem{isotropic} B. Andrews, {\em Classification of limiting shapes for isotropic curve flows}, J. Amer. Math. Soc. {\bf 16} (2003) no. 2, 443-459.
\bibitem{BC} B. Andrews, X. Chen, {\em Surfaces moving by powers of Gauss curvature}, Pure Appl. Math. Q. {\bf 8} (2012) no. 4, 825-834.
\bibitem{AGN}B. Andrews, P. Guan, and L. Ni, {\em Flow by powers of the Gauss curvature}, Adv. Math. {\bf 299} (2016), 174-201. 
\bibitem{Andrews-Han-Li-Wei}
B. Andrews, X. Han, H. Li, Y. Wei, Non-collapsing for hypersurface flows in the sphere and hyperbolic space, Ann. Sc. Norm. Super. Pisa Cl. Sci. (5) {\bf 14}(2015), no.1, 331-338.
\bibitem{BCD}S. Brendle, K. Choi, P. Daskalopoulos, {\em Asymptotic behavior of flows by powers of the Gaussian curvature}, Acta Math. {\bf 219} (2017), no. 1, 1-16.
\bibitem{CGH}M. Chen, P. Guan, J. Huang, {\em In-homogeneous Gauss curvature type flows}, work in progress.
\bibitem{Choi}K. Choi, P. Daskalopoulos, {\em Uniqueness of closed self-similar solutions to the Gauss curvature flow}, arXiv: 1609.05487.
\bibitem{CW} K. Chou and X. Wang, {\em A logarithmic Gauss curvature flow and the Minkowski problem}, Ann. Inst. H. Poincaré Anal. Non Linéaire {\bf 17} (2000), no. 6, 733-751.
\bibitem{C nroot}B. Chow, {\em Deforming convex hypersurfaces by the nth root of the Gaussian curvature}, J. Diff. Geom. {\bf 22} (1985), no. 1, 117-138.
\bibitem{C entropy}B. Chow, {\em On Harnack's inequality and entropy for the Gaussian curvature flow}, Comm. Pure Appl. Math. {\bf 44} (1991), no. 4, 469-483.
\bibitem{dW}M. P. do Carmo, F. W. Warner, {\em Rigidity and convexity of hypersurfaces in spheres}, J. Diff. Geom. {\bf 4} (1970), 133-144.
\bibitem{F}W.J. Firey, {\em Shapes of worn stones}, Mathematika {\bf 21} (1974), 1-11.
\bibitem{Gage-Hamilton}M. Gage, R. Hamilton, {\em The heat equation shrinking convex plane curves}, J. Diff. Geom., {\bf 23} (1986), no. 1, 69-96.
\bibitem{Gerhardt book} C. Gerhardt, {\em Curvature problems}, Series in Geometry and Topology, 39. International Press, Somerville, MA, 2006.
\bibitem{Gerhardt sphere} C. Gerhardt, {\em Curvature flows in the sphere}, J. Diff. Geom. {\bf 100} (2015), no. 2, 301-347. 
\bibitem{Grayson}M. Grayson, {\em The heat equation shrinks embedded plane curves to round points}, J. Diff. Geom. {\bf 26} (1987), no. 2, 285-314.
\bibitem{Guan-Ni} P. Guan and L. Ni, {\em Entropy and a convergence theorem for Gauss curvature flow in high dimension}, J. Eur. Math. Soc.  {\bf 19} (2017), no. 12, 3735-3761.
\bibitem{Guan-Li}
P. Guan and J. Li, {\em A mean curvature type flow in space forms}, Int. Math. Res. Not., (2015), no. 13, 4716-4740.
\bibitem{GLW}Q. Guang, Q. Li, and X. Wang, {\em The Minkowski problem in the sphere}, preprint. 
\bibitem{H}R. Hamilton, {\em Remarks on the entropy and Harnack estimates for the Gauss curvature flow}, Comm. Anal. Geom. {\bf 2} (1994), no. 1, 155-165.
\bibitem{Huisken-R}
G. Huisken, {\em contracting convex hypersurfaces in Riemannian manifolds by their mean curvature}, Invention. Math. {\bf 84}(1986), no.3, 463-480.
\bibitem{McCoy}J. McCoy, {\em Curvature contraction flows in the sphere}, Proc. Am. Math. Soc. {\bf 146} (2018), 1243-1256.
\bibitem{ST} G. Sapiro, A. Tannenbaum, {\em On affine plane curve evolution}. J. Funct. Anal. (1994) no.1, 79-120.
\bibitem{Schneider}R. Schneider, {\em Convex bodies: the Brunn-Minkowski Theory}, Encyclopedia Math. Appl. 151, Cambridge Univ. Press, 2014. 
\bibitem{Krylov} N. V. Krylov, {\em Nonlinear elliptic and parabolic equations of the second order}, Translated from the Russian by P. L. Buzytsky [P. L. Buzytskiĭ]. Mathematics and its Applications (Soviet Series), 7. D. Reidel Publishing Co., Dordrecht, 1987.
\bibitem{Tso} K. Tso, {\em Deforming a hypersurface by its Gauss-Kronecker curvature}, Comm. Pure Appl. Math. {\bf 38} (1985), no. 6, 867-882.






\end{thebibliography}
\end{document}